\documentclass{cmslatex}
\usepackage[paperwidth=7in, paperheight=10in, margin=.875in]{geometry}
 \usepackage[backref,colorlinks,linkcolor=red,anchorcolor=green,citecolor=blue]{hyperref}
\usepackage{amsfonts,amssymb}
\usepackage{amsmath}
\usepackage{graphicx}
\usepackage{cite}
\usepackage{cleveref}
\usepackage{enumerate}
\renewcommand{\theequation}{\arabic{section}.\arabic{equation}}

\usepackage{comment}
\usepackage[utf8]{inputenc} %
\usepackage[T1]{fontenc}    %
\usepackage{url}            %
\usepackage{hyperref}
\usepackage{booktabs}       %
\usepackage{amsfonts}       %
\usepackage{nicefrac}       %
\usepackage{microtype}      %
\usepackage{lipsum}
\usepackage{amsmath,amssymb,mathrsfs}
\usepackage{bbm}
\usepackage{tcolorbox}
\usepackage{enumitem}
\usepackage{mathtools}
\usepackage{latexsym}
\usepackage{multimedia}
\usepackage{media9}
\usepackage{relsize}
\usepackage{algorithm}
\usepackage{algorithmic}
\usepackage{lineno}
\usepackage{tikz}
\usepackage{enumitem}
\usetikzlibrary{arrows.meta}
\usepackage{comment}
\usepackage{multicol}
\usepackage{graphicx}
\usepackage{subfig}
\usepackage{xcolor}
\usepackage{todonotes}

\DeclareMathOperator*{\argmin}{argmin}

\newcommand{\bq}{\begin{equation}}
\newcommand{\eq}{\end{equation}}
\newcommand{\bbP}{{\mathbb P}}
\newcommand{\R}{\mathbb{R}}
\newcommand{\N}{\mathbb{N}}
\newcommand{\Prob}{\mathbb{P}}

\newcommand{\cA}{\mathcal{A}}
\newcommand{\cB}{\mathcal{B}}
\newcommand{\cC}{\mathcal{C}}
\newcommand{\cD}{\mathcal{D}}

\newcommand{\fkf}{\mathfrak{f}}

\newcommand{\myin}{\Omega_n}
\newcommand{\myout}{\Omega^c_n}

\newcommand{\fX}{\mathfrak{X}}

\graphicspath{{Fig/}}

\newcommand{\TheShortTitle}{AdaVar for Global Optimization}
\newcommand{\TheAuthors}{B.\ Engquist, K.\ Ren and Y.\ Yang}

\allowdisplaybreaks
\begin{document}

\title{An Algebraically Converging Stochastic Gradient Descent Algorithm for Global Optimization\thanks{Received date, and accepted date (The correct dates will be entered by the editor)}}

\author{
  Bj\"orn Engquist\thanks{Department of Mathematics and the Oden Institute, The University of Texas, Austin, TX 78712;
  \texttt{engquist@oden.utexas.edu}} 
   \and Kui Ren\thanks{Department of Applied Physics and Applied Mathematics, Columbia University, New York, NY 10027; \texttt{kr2002@columbia.edu}}
  \and Yunan Yang\thanks{Department of Mathematics, Cornell University, Ithaca, NY 14850;
 \texttt{yunan.yang@cornell.edu}} 
}

\pagestyle{myheadings}
\markboth{\TheShortTitle}{\TheAuthors} 
\maketitle
         
\maketitle

\begin{abstract}
We propose a new gradient descent algorithm with added stochastic terms for finding the global optimizers of nonconvex optimization problems. A key component in the algorithm is the adaptive tuning of the randomness based on the value of the objective function. In the language of simulated annealing, the temperature is state-dependent. With this, we prove the global convergence of the algorithm with an algebraic rate both in probability and in the parameter space. This is a significant improvement over the classical rate from using a more straightforward control of the noise term. The convergence proof is based on the actual discrete setup of the algorithm, not just its continuous limit as often done in the literature. We also present several numerical examples to demonstrate the efficiency and robustness of the algorithm for reasonably complex objective functions.
\end{abstract}



\section{Introduction}\label{sec:intro}

The gradient descent method has long been used for the optimization of real-valued functions. For minimization of function $f(x)\in C^1$ and $f:\R^d \mapsto \R$, the algorithm takes the form
\begin{equation}\label{eq:GD101}
X_{n+1} = X_n -\eta_n G(X_n),
\end{equation}
where $G(x) = \nabla f(x)$, $\eta_n$ is the step size, and $\{X_n\}$ are the iterates. Under proper conditions, the gradient descent algorithm converges to a local minimum within the basin of attraction. A sufficient convergence condition for strongly convex functions is $\eta_n = \eta < 2/L$, where $L$ is the Lipschitz constant for $G(x)$. We will also satisfy this condition in the stochastic setting proposed below. When applied to general functions with many local minima, there is no guarantee of convergence to a global minimum through gradient descent~\eqref{eq:GD101}.

For about half a century, an additional random term has been proposed for~\eqref{eq:GD101} to achieve global convergence for nonconvex functions. It results in a stochastic gradient descent algorithm
\begin{equation}\label{eq:SGD101}
    X_{n+1} = X_n -\eta_n G(X_n) + \sigma_n  \psi_n,
\end{equation}
where $\{\psi_n\}$ are random variables following a certain distribution, and we will assume $\psi_n \sim \mathcal{N}(0,I)$, the standard $d$-dimensional Gaussian, hereafter. The standard deviation $\sigma_n$ controls the strength of the noise. Equation~\eqref{eq:SGD101} represents the type of algorithm we focus on in this paper.

There are also many stochastic gradient descent techniques where the function and its gradient are intrinsically stochastic or where the gradient is stochastic from sampling, as is often the case in machine learning, e.g.,
\begin{equation}\label{eq:MLSGD101}
    X_{n+1} = X_n -\eta_n G(X_n;\xi_n),
\end{equation}
where $G(x;\xi)$, parameterized by the random variable $\xi$, is a stochastic approximation to the true gradient $G(x)$. %
In this paper, we do not address algorithms on the common form~\eqref{eq:MLSGD101} but will provide a numerical example in~\Cref{sec:numerics} where the adaptive variance is achieved by varying the batch size and offer a brief discussion on it in~\Cref{sec:conclusion} under future directions. We also mention that the form~\eqref{eq:SGD101} is sometimes used in the analysis of algorithm~\eqref{eq:MLSGD101} with $G(X_n)$ as the expected value of $G(X_n, \xi_n)$.

In an algorithm of the form~\eqref{eq:SGD101}, which is the focus of this work, there is an obvious dilemma. The sequence $\{\sigma_n\}$ has to converge to zero as $n$ increases even for convergence applied to convex functions. On the other hand, if $\sigma_n$ decreases too fast, the convergence will be to a local minimum and not the global one. This is analyzed in the classical papers~\cite{geman1986diffusions,chiang1987diffusion,kushner1987asymptotic,hwang1990large,gelfand1991recursive} starting from the late 1980s, and the critical decay rate is $\mathcal{O}(\frac{1}{\sqrt{\log n}})$. The analyses were mainly based on the asymptotic limit as a stochastic differential equation related to the scaling,
\begin{equation}\label{eq:diffdisc}
    X_{n+1} = X_n - \eta_n G(X_n) +  \sqrt{2\eta_nT_n} \psi_n,
\end{equation}
where $T_n\in\R^+$ is often referred to as the temperature parameter.

To improve on this slow convergence rate, we propose an algorithm with an adaptive \textit{state-dependent} $\sigma_n$,
\begin{equation}\label{eq:diffdisc_new}
    X_{n+1} = X_n - \eta_n G(X_n) +  \sigma_n \left(f(X_n) \right) \psi_n.
\end{equation}
This type of algorithm with an adaptive variance, which we label as ``AdaVar'' for future reference, can have a larger variance away from the minimum to explore the overall landscape and avoid trapping in a local minimum. It can then have a smaller variance for lower $f$-values, typically closer to the minimum. We will here use the simple form of the variance below based on the cutoff scalar value $f_n$ with one variance for $f$-values above $f_n$ and a smaller one for the case below the cutoff $f_n$. The sequence of scalar values $\{f_n\}$ are chosen by the user to implement the algorithm. That is,
\begin{equation}\label{eq:two-stage-sigma}
 \sigma_n( f(X_n) ) =    
 \begin{cases}
  \sigma_n^-, & f(X_n) \leq f_n,\\
   \sigma_n^+, & f(X_n) > f_n.
    \end{cases}
\end{equation}
We also assume that $f$ has a Lipschitz continuous gradient and a unique global minimum. Furthermore, there should be an open set containing the global minimum where $f$ is strongly convex. Given a current iterate $X^{\text{c}}$ where $f(X^\text{c}) > f_n$, the law of the following iterate $X^{\text{f}}$ is a normalized Gaussian given a bounded search domain, centered at $X^{\text{c}} - \eta_n G(X^{\text{c}})$, with the standard deviation $\sigma_n^+$. Our convergence theory in Section~\ref{sec:main} is based on a limiting case where $\sigma_n^+ \rightarrow \infty$ for all $n$ for which the law of the following iterate $X^{\text{f}}$ becomes the uniform distribution. Our numerical illustration in Section~\ref{sec:numerics} is based on~\eqref{eq:two-stage-sigma}.

Our main goal is to prove that there exists a stochastic gradient descent algorithm of this type, which has a much faster algebraic convergence rate in both probability and space than the classical logarithmic rate; see Theorem~\ref{thm:main}  below. We also prove that this algebraic rate in probability is optimal; see Proposition~\ref{prop} below. We want to stress that the convergence proof is for the true discrete algorithm and not just for the continuous limit via stochastic differential equations, which is often the case in the literature. To have a clean proof, we assume certain properties of the objective function to be known, for example, the volume of the sub-level sets. Since that is not always available, we describe in~\Cref{sec:extension} how these properties can be estimated during the optimization process. We further show that the exact properties are not required.

There are numerous similar strategies for global convergence, and many of them use multiple sequences of $X_n$ values~\cite[Section 1]{zhigljavsky2007stochastic}. In the simulated annealing algorithm~\cite{kirkpatrick1983optimization}, all new points that lower the objective function are accepted, and, with a certain probability, points that raise the objective function value are also accepted. For simulated annealing, the global convergence (in probability)  is achieved as long as the iterates can visit the entire search domain, and can be proved for the discrete algorithm~\cite{locatelli1996convergence}. In parallel tempering and its variants, such as the replica exchange method~\cite{dong2021replica}, Metropolis type of acceptance-rejection strategies are used to alternate between different sequences of $X_n$. 
When the multimodality of the objective function is mild, methods known as ``multistart''~\cite[Section 2.6]{zhigljavsky2007stochastic} and ``random initialization''~\cite{chen2019gradient} can be efficient in finding the global minimum. Stochastic gradient descent with noise of machine
learning type can achieve exponential convergence to the global minimum for objective functions satisfying certain conditions~\cite{wojtowytsch2023stochastic}.

In consensus-based optimization~\cite{carrillo2021consensus,totzeck2021trends}, a group of particles explore the optimization landscape simultaneously under some stochastic influence and then build a consensus at the position of the weighted mean that is located near the global minimizer. Swarm-based gradient descent methods with randomness can also handle non-convex optimization problems~\cite{lu2022swarm,tadmor2023swarm}, in which agents communicate and then change both the mass and the location to achieve both exploration and exploitation. There has also been work on global optimization using radial basis function interpolation from a history of carefully sampled iterates~\cite{bull2011convergence}. In a recent work~\cite{li2023convergence}, an algebraic convergence in both space and probability is achieved by using an $\epsilon$-greedy algorithm to obtain the iterates. In classical genetic algorithms (see the nice book~\cite{yang2020nature} for several examples), global convergence is achieved by random seeding such that one of the locally converging sequences will converge to the global minimum. Optimal control problems aiming at selecting optimal tempering schemes to reduce computational cost have also been studied~\cite{gao2020state}. The algorithm proposed here can also be used in parallel but is powerful even as a single sequence method, as seen in~\Cref{sec:example1_10D}. The basin of attraction for the global minimum has a volume proportionally $10^{-7}$ of the volume of the entire domain, but the implicit seeding from the AdaVar algorithm practically converges in $\mathcal{O}(10^4)$ iterations, which is much faster than what random seeding could give.

The rest of the paper is organized as follows. The mathematical setup and the main result of the paper are presented in~\Cref{sec:main}. We then prove the main theorem in~\Cref{sec:proof} through several lemmas. In~\Cref{sec:extension}, we discuss algorithmic details in terms of the implementation of the proposed algorithm and several key parameter estimations. In~\Cref{sec:numerics}, we present numerical examples illustrating the effectiveness of adaptive state-dependent variance. This includes a few illustrations that are not mathematically analyzed in the current study. One such instance pertains to the kind of optimization frequently seen in machine learning. Meanwhile, another focuses on optimization without gradients, a topic that is thoroughly examined in a separate work~\cite{engquist2023adaptive}. These extensions are further discussed as part of~\Cref{sec:conclusion}.

\section{Statement of Main Result}\label{sec:main}

While the decay $\mathcal O(\frac{1}{\sqrt{\log n}})$ in $\sigma_n$ in~\eqref{eq:SGD101} ensures the convergence of the iterates $X_n$ to the global minimizer $x^*$ in probability~\cite{hwang1990large}, it only gives a logarithmic convergence rate.
In this paper, we propose a new SGD algorithm with an \textit{algebraic} rate of convergence. The faster convergence here is achieved by having an adaptive and state-dependent noise term in the form of~\eqref{eq:diffdisc_new}.

Let us consider a nonconvex objective function $f:\R^d  \mapsto \R$, $f\in C^1$ and $ G(x) = \nabla f(x)$. We have the following additional assumptions for the function $f(x)$.
\begin{enumerate}[label=\textbf{A\arabic*}]
  \item \label{itm:A1} On a bounded open set $\Omega_{sc} \subset \R^d$,  $f \in C^2(\Omega_{sc})$ is strongly convex, and $b_1 I \preceq \nabla^2 f(x) \preceq b_2 I$, where $0< b_1 \leq b_2 <\infty$, $x\in \Omega_{sc}$, and $I\in\R^{d\times d}$ is the identity matrix;
  \item \label{itm:A2} $f(x)$ has a unique global minimizer $x^*$ that lies strictly inside $\Omega_{sc}$ and $f(x^*) = \fkf^*$. 
\end{enumerate}
We follow the two-stage adaptive variance~\eqref{eq:two-stage-sigma} and consider the following discrete algorithm for minimizing $f$. Let $\mathfrak{X}\subset \R^d$ be a non-empty compact set such that $\Omega_{sc} \subseteq \mathfrak{X}$. Given any $X_0 \in \mathfrak{X}$, we study an iterative optimization process for $n\geq 1$,
\begin{equation}\label{eq:diffdisc1}
    X_{n+1} = \begin{cases}  X_n - \eta_n G(X_n) + \sigma_n \psi_n, & X_n \in \Omega_n,\\
    \phi_{n}, &X_n  \in \Omega^c_n= \R^d \setminus \Omega_n, \end{cases}
\end{equation}
where $\sigma_n\in \R^+$ is the standard deviation of the Gaussian random term, and $\{\phi_{n}\}$ are i.i.d.\ random variables following $\mathcal{U}(\mathfrak{X})$, the uniform distribution on $\mathfrak{X}$. As mentioned in Section~\ref{sec:intro}, the random schedule in~\eqref{eq:diffdisc1} conditioned on a fixed $X_n$ is the limiting case of~\eqref{eq:two-stage-sigma} if $\sigma_n^+ \rightarrow \infty$, but it may not be the limiting dynamics of~\eqref{eq:diffdisc_new}. The random schedule~\eqref{eq:diffdisc1} shares the flavor of a stochastic gradient
descent algorithm with random restarts in the domain triggered by the objective function's value.

The subset $\Omega_n\subset \R^d$, $\forall n\in \N$, is the ${f}_n$-sublevel set of function $f(x)$ defined as
\begin{equation}\label{eq:omega_n_and_f_n}
    \Omega_n = \{x\in \R^d: f(x) \leq {f}_n\},
\end{equation}
where the real sequence $\{f_n\}$ monotonically decreases as $n$ increases with $f_n\xrightarrow{n\rightarrow \infty} f(x^*)$. Note that the sublevel sets satisfy $\Omega_{n_1} \subseteq \Omega_{n_2}$ if $n_1\geq n_2$. Without loss of generality, we assume that $\Omega_0\subset \mathfrak{X}$ for the chosen starting iterate $X_0$, and consequently, we have $\Omega_n\subset \mathfrak{X},\,\forall n\in\mathbb{N}$. The set $\mathfrak X$ can be considered as the target search domain as~\eqref{eq:diffdisc1} enforces the iterates $\{X_n\}$ to always come back to $\mathfrak X$. 

Let us emphasize that~\eqref{eq:diffdisc1} is simply the limiting case of~\eqref{eq:diffdisc_new} together with the two-stage adaptive variance~\eqref{eq:two-stage-sigma} when $\sigma_n^+$ is set to be $+\infty$.

We will use the following notations consistently throughout the paper. Given a measurable set $A\subset \R^d$ with a nonzero Lebesgue measure, i.e., $|A|> 0$, 
\begin{align}
     X_n^+ &:= X_n -\eta_nG(X_n), \nonumber \\ %
   k_n(y,x) &:=  \left(\sqrt{2\pi } \sigma_n\right)^{-d}    \exp \bigg(-\frac{|y-x^+ |^2}{2\sigma^2_n}\bigg),\quad x^+ =  x -\eta_n G(x), \label{eq:kernel} \\    
  p_{n}(A) &:= \Prob ( X_{n+1} \in A |X_n\in\myin ) = \frac{\int_{\Omega_n} \left( \int_A k_n(y,x) dy \right) d\mu_n(x) }{\int_{\Omega_n} d\mu_n(x)} , \label{eq:p_n} \\
       q_{n}(A) &:= \Prob ( X_{n+1}\in A |X_n\in\myout ) = \frac{|A\cap \mathfrak{X}|}{|\mathfrak{X}|}, \label{eq:q_n}
\end{align}
where $\mu_n$ is the probability distribution associated with the random variable $X_n$. Thus, $p_n$ and $q_n$ represent the conditional probability given $X_n\in\myin$ and $X_n\in\myout$, respectively. 

The pinnacle of the proposed algorithm is how we control the decay rates of $\sigma_n$ and $|\Omega_n|$ simultaneously such that
\begin{enumerate}
    \item $\sigma_n\rightarrow 0$ as $n\rightarrow 0$; That is, the noise variance upon the trusted region decreases.
    \item $|\Omega_n|\rightarrow 0$ as $n\rightarrow \infty$, due to the monotonic decrease of ${f}_n$.
\end{enumerate}
We add a few more assumptions that concretely determine the iteration in~\eqref{eq:diffdisc1}, based on which the main theorem holds. Let $c_0$ be the radius of a ball in $\R^d$ with unit volume.
\begin{enumerate}[label = \textbf{B\arabic*}]
\item \label{itm:B1}The step size in~\eqref{eq:diffdisc1} is a constant where $\eta_n = \eta < 2/b_2$.
\item \label{itm:B2} For $n\geq 1$, ${\sigma}_{n} = \dfrac{c_0 \sqrt[d]{|\Omega_n|}}{\sqrt{ \log n}}$ where $\Omega_n$ depends on $f_n$ based on~\eqref{eq:omega_n_and_f_n}. %
\item \label{itm:B3} The scalar value $f_0$ is chosen such that $\Omega_0 \in \frak{X}$ and the sequence $\{f_n\}$ is chosen such that $|\Omega_{n+1}| = |\Omega_0| n^{-\alpha}$, $\forall n\geq 1$, where $0<\alpha <  \alpha^* = \frac{(c^*)^2}{2}$, $c^* = \frac{2\eta b_2 - \eta^2 b_2^2}{4} \left( \frac{b_1}{b_2}\right)^{\frac{3}{2}}$; $|\Omega_1| = |\Omega_0|$.

\end{enumerate}
We remark that if $f_n \equiv f_0$ for all $n$, we have $\Omega_n \equiv \Omega_0$. Consequently, \ref{itm:B2} becomes the classical decay rate $\sigma_n = \mathcal{O} (\frac{1}{\sqrt{\log n}})$. We used a constant step size $\eta = 1/b_2$ in the proofs for simplicity.

Our proposed algorithm~\eqref{eq:diffdisc1} also reduces to the classical case~\eqref{eq:diffdisc} for which the $\mathcal{O} (\frac{1}{\sqrt{\log n}})$ decay rate is optimal to guarantee global convergence~\cite{geman1986diffusions,kushner1987asymptotic,hwang1990large,hu2019diffusion}. Besides, if $\eta = 1/b_2$, the upper bound $\alpha^*$ for $\alpha$ in~\ref{itm:B3} is maximized. We also stress that the formula in~\ref{itm:B3} does not add any extra constraints on the objective function $f(x)$. It is only used to define the sequence $\{f_n\}$. In~\Cref{sec:extension}, we will discuss several strategies to estimate $\{f_n\}$ if unknown a priori.

We are ready to state the main theorem of this paper in the form of a concentration inequality.
\begin{theorem}\label{thm:main}
Let \ref{itm:A1}-\ref{itm:A2} and \ref{itm:B1}-\ref{itm:B3} hold and the iterates $\{X_n\}$ follow the stochastic gradient descent algorithm~\eqref{eq:diffdisc1}. 
Then, $\exists N\in \N$ such that $\forall n\geq N$,
\[
\Prob \left( |X_n-x^*|  > \mathfrak{c}_1 n^{-\alpha/d} \right) \leq \mathfrak{c}_2 n^{\alpha-\alpha_2},
\]
where $\alpha_2 = \frac{\alpha}{2} + \frac{(c^*)^2}{4}$. Here, $\mathfrak{c_1}, \mathfrak{c_2}$ are constants that depend on $b_1$ and $b_2$.
\end{theorem}
The main goal here is to prove that there is an algorithm of AdaVar type with algebraic convergence rates. The algorithm depends on the volume of $\Omega_n$ as a function of $f_n$ and the parameters $b_1$ and $b_2$. From the proof in~\Cref{sec:proof} and the analysis in~\Cref{sec:extension}, it is clear that these estimates can be very rough, and we would still have algebraic convergence but potentially with a smaller $\alpha$.

\begin{remark}
We remark that there are two situations where the convergence in~\Cref{thm:main} can be improved to an almost-sure convergence.
\begin{itemize}
\item If we know  a value $f_{sc} \in \mathbb{R}$ such that 
    $\{x: f(x) \leq f_{sc} \} \subseteq \Omega_{sc}$, once an iterate $X_n$ satisfies $f(X_n) \leq f_{sc} $, this iterate has entered the global basin of attraction over which the objective function is strongly convex. We can then turn off the random component of the algorithm. This leads to the following variant of the algorithm in~\eqref{eq:diffdisc1}
\begin{equation*}  
     X_{n+1} = \begin{cases}  X_n - \eta_n G(X_n), & f(X_n) \leq f_{sc}  ,\\
     X_n - \eta_n G(X_n) + \sigma_n \psi_n, & f_{sc} < f(X_n) \leq f_n,\\
    \phi_{n}, &  f(X_n) >  f_n, \end{cases}
\end{equation*}
where $\Omega_n$ is the $f_n$-sublevel defined in~\eqref{eq:omega_n_and_f_n}.   
Once $X_n$ has entered the set $\{x\in \Omega: f(x) \leq f_{sc}\}$, the scheme consequently becomes a (deterministic) gradient descent method. Our assumption~\ref{itm:B1}  on the step size ensures the exponential convergence of $X_n$ to $x^*$ in terms of their Euclidean distance. Therefore, in the context of the proposal stochastic algorithm, we will have an almost-sure convergence to the global minimizer $x^*$.

\item  If we extract a different set of iterates $\{ \Tilde{X}_n \}$ where
\[
\Tilde{X}_n  = \argmin_{0\leq k \leq n} f(X_k),
\]
then $\Tilde{X}_n \rightarrow x^*$ almost surely due to Property~\ref{itm:E1} since it is an almost-sure event that at least one of the iterates from $\{X_n\}_{n=1}^\infty$ visits every set of nonzero Lebesgue measure.
\end{itemize}
Our main result (Theorem~\ref{thm:main}) does not work with either of the scenarios and has a convergence in probability result instead.
\end{remark}

There are three main goals in the rest of the paper. The first 
is to prove in~\Cref{sec:proof} the existence of the algebraically converging algorithm of this type as stated in~\Cref{thm:main}. This is done by an explicit algorithm under the assumption that we know the value of $f(x^*) = \fkf^*$, and the volume function of sub-level sets $\Omega^*(\fkf) = |\{x\in \mathfrak{X}|f(x)\leq \fkf \}|$, whose inverse exists. Given the volume $|\Omega_n|$, we can locate the cutoff values $f_n$ based on~\ref{itm:B3}.
After proving the results under the assumptions mentioned above, we consider the practical application of the algorithm and give guidance 
in~\Cref{sec:extension} for how to estimate $f(x^*)$ and the sub-level set volume function $\Omega^*$ (and more importantly, its inverse function), which we often do not know apriori. Third, we present several numerical examples in~\Cref{sec:numerics} to demonstrate results from the previous two parts.

In our proof of~\Cref{thm:main}, we mainly show that our algorithm, like many other algorithms in this category, achieves global convergence by balancing the ``exploration'' (visiting the entire search domain) and ``exploitation'' (concentrating around the global minimum) phases. More precisely, we show that the sequence $\{X_k\}_{k\ge 0}$ generated by the algorithm satisfies the following two properties:
\begin{enumerate}[label = \textbf{E\arabic*}]
    \item\label{itm:E1} (Exploration)  For any subset $\Omega \subset \fX$ with non-zero Lebesgue measure, one of the history iterates will visit $\Omega$ almost surely. That is, 
$$
\bbP\left(\bigcup_{k=0}^\infty \{X_k \in \Omega\} \right) = 1.
$$
    \item\label{itm:E2} (Exploitation) For any $\epsilon>0$, we have that
    $$p_n^\epsilon : = \bbP(|X_n - x^*| > \epsilon) \xrightarrow{n\rightarrow \infty} 0\,.$$
    \end{enumerate}
Moreover, we show that, for our algorithm, $p_n^\epsilon$ converges to $0$ at an algebraic rate. In our main proof of~\Cref{thm:main}, we proved the ``exploration'' (Property 1) in~\Cref{sec:prop1} and the ``exploitation'' (Property 2) in~\Cref{sec:prop2}. 

It turns out that, any stochastic optimization algorithm that achieves global convergence with only~\ref{itm:E1} and~\ref{itm:E2}, i.e., without utilizing more specific structures of the objective function $f$, can not have a convergence rate that is faster than $\mathcal{O}(n^{-1})$, as summarized in the following proposition.

\begin{prop}\label{prop}
     Let $\{X_k\}_{k\ge 0}$ be a sequence that satisfies~\ref{itm:E1} and~\ref{itm:E2}. Then
     \[
        n^{1+\gamma} p_n^\epsilon \xrightarrow{n\rightarrow \infty} +\infty, \qquad \forall \gamma>0\,.     
     \]
\end{prop}
\begin{proof}
Assume otherwise, then the sequence $\{p_n^\epsilon\}$ converges to $0$ faster than $\mathcal{O}(n^{-1})$. Therefore the series $\sum_{k=0}^{\infty } p_k^\epsilon$ converges. This means that $\exists N \in \mathbb{N}$ such that
\[
    \sum_{k=N}^{\infty } p_k^\epsilon   \le \frac{1}{2}\,.
\]
For this $N$, and a sufficiently small $\delta>0$, let $\Omega$ be a subset of $\fX \setminus \left( \cup_{k=0}^{N-1} B_\delta(X_k) \cup B_\epsilon(x^*) \right)$ where $B_\delta(X_k)$ (resp $B_\epsilon(x^*)$) is the open ball of radius $\delta$ (resp $\epsilon$) centered at $X_k$ (resp $x^*$). Then
\[
X_0,\ldots, X_{N-1} \not \in \Omega\,.
\]
For this particular choice of $\Omega$, we have that 
\begin{eqnarray*}
 \bbP\left(\bigcup_{k=0}^\infty \{X_k \in \Omega\} \right) & = & \lim_{n\rightarrow \infty} \bbP\left( \{X_{n} \in \Omega\} \cup\{ X_{n-1} \in \Omega\} \cup \ldots\cup   \{ X_{N} \in \Omega\}  \cup \ldots \cup \{ X_{0} \in \Omega\} \right)\\
 &=&  \lim_{n\rightarrow \infty} \bbP\left( \{X_{n} \in \Omega\} \cup\{ X_{n-1}\in \Omega\} \cup \ldots\cup \{ X_{N} \in \Omega\} \right)\\
 &\leq &
 \sum_{k=N}^{\infty} \bbP\left( \{ X_{k} \in \Omega\}  \right)\\
 &\leq &  \sum_{k=N}^{\infty}  \bbP\left( \{ X_{k} \not\in B_{\epsilon}(x^*)\}  \right)  = 
   \sum_{k=N}^{\infty}  p_k^\epsilon \le  \frac{1}{2}.
\end{eqnarray*}
This contradicts the  ``exploration'' condition in~\ref{itm:E1}. %
\end{proof}

\section{Convergence Proof}\label{sec:proof}

The proof of~\Cref{thm:main} is divided into three parts. In~\Cref{sec:prop1}, we first present several lemmas demonstrating that the iterates following algorithm~\eqref{eq:diffdisc1} will visit any subset of the domain of interest with a nonzero measure almost surely, under the assumptions in~\Cref{sec:main}. In~\Cref{sec:prop2}, we present another three lemmas to obtain the lower and upper bounds for the key conditional probability $p_n(\Omega_{n+1}^c)$ when $n$ is sufficiently large. Finally, in~\Cref{sec:main_proof}, we combine all the previous lemmas and prove that the iterates converge to the global minimizer in probability with concrete algebraic convergence rates.

\subsection{Property One: Full Coverage of the Entire Domain}\label{sec:prop1}

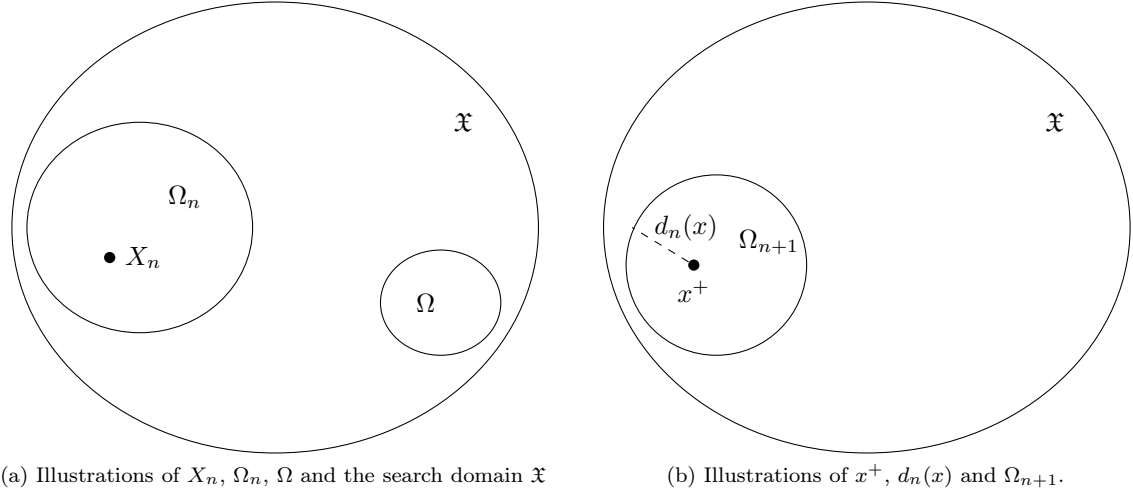
\begin{figure}
\centering
\subfloat[Illustrations of $X_n$, $\Omega_n$, $\Omega$ and the search domain $\mathfrak X$]{
\scalebox{0.8}{
\begin{tikzpicture}%
        \draw[] (0,0) ellipse (3.5 and 3);
        \draw[] (-1.8,0) ellipse (1.5 and 1.4);
        \draw[] (2.2,-1) ellipse (0.8 and 0.7);
		\node [] (7) at (-1.2, 0.4) {$\Omega_n$};
		\node [] (8) at (2.5, 1.4) {$\mathfrak{X}$};
        \node [] (12) at (2, -1.0) {$\Omega$};
        \draw [] (-2.2, -0.4) node[circle,fill,inner sep=1.5pt,label=right:$X_n$](a){}; 
\end{tikzpicture}} \label{fig:omega1} 
}
\hspace{0.3cm}
\subfloat[Illustrations of $x^+$, $d_n(x)$ and $\Omega_{n+1}$.]{
\scalebox{0.8}{
\begin{tikzpicture}%
		\draw [] (-2.3, -0.5) node[circle,fill,inner sep=1.5pt,label=below:$x^+$](a){}; %
        \draw [dashed] (a){}--node[above]{}(-3.12,0.0);
		\draw (0,0) ellipse (3.5 and 3);
		\draw (-2,-0.5) ellipse (1.2 and 1.2);
        \node [] (12) at (-1.3, -0.2) {$\Omega_{n+1}$};
        \node [] (12) at (-2.4, 0.0) {$d_n(x)$};
        \node [] (9) at (2.5, 1.4) {$\mathfrak{X}$};
\end{tikzpicture}} \label{fig:omega_n}
}
\caption{Illustrations for proofs of Property One in Section~\ref{sec:prop1} and Property Two in Section~\ref{sec:prop2}. The set $\Omega$ in (a) is used in all lemmas in Section~\ref{sec:prop1}. The distance $d_n(x)$ in (b) is defined in~\eqref{eq:d_n def}. 
}
\end{figure}

Given the initial iterate $X_0$, we are interested in finding the global minimum within the set $\mathfrak{X}$. For an arbitrary measurable set $\Omega \subseteq \mathfrak{X}$ where $|\Omega|\neq 0$, we use the following shorthand notations hereafter:
\begin{multicols}{2}
\begin{itemize}
    \item the event $X_{n} \in \Omega_n$ as $A_n$,
    \item the event  $X_{n} \not \in \Omega_n$ as $A_n^c$,
    \item the event $X_{n} \in \Omega$ as $B_n$,
    \item the event $X_{n} \not\in \Omega$ as $B_n^c$. 
\end{itemize}\end{multicols}
\noindent See~\Cref{fig:omega1} for illustrations of $\Omega$, $\Omega_n$ and related sets.  We start with the following lemma.

\begin{lemma}\label{lem:ball}
Given any $x\in\R^d$ and $\sigma, V \in\R^+$, let $\mathcal{M}_{V} = \{A \subseteq \R^d: |A| = V \}$. For any set $A\in\mathcal{M}_{V}$, we have
\begin{equation}\label{eq:ball}
  \left(\sqrt{2\pi } \sigma\right)^{-d}  \int_{A} \exp \bigg(-\frac{|y-x|^2}{2\sigma^2}\bigg) dy \leq   \left(\sqrt{2\pi } \sigma\right)^{-d} \int_{B(x;R)} \exp \bigg(-\frac{|y-x|^2}{2\sigma^2}\bigg) dy,
\end{equation}
where $B(x;R)\in \mathcal{M}_{V}$ is a ball centered at $x$ with radius $R = c_0 \sqrt[d]{V}$, $c_0 =  \pi^{-\frac{1}{2}}\sqrt[d]{\Gamma \left({\frac {d}{2}}+1\right)}$ being the radius for a ball of unit volume in $\R^d$.
\end{lemma}

\begin{proof}
First, recall that the volume of a $d$-dimensional ball of radius $R$ is $\frac {(\sqrt{\pi}R)^d}{\Gamma \left({\frac {d}{2}}+1\right)}$,
where $\Gamma(s)$ is the Gamma function. Thus, a $d$-dimensional ball of radius $R = c_0 \sqrt[d]{V}$ has volume $V$.

Next, we define a function $f(x,y;\sigma) =  \left(\sqrt{2\pi } \sigma\right)^{-d} e^{-\frac{|y-x|^2}{2\sigma^2}}$. Given any $A\in \mathcal{M}_{V}$, 
\begin{eqnarray}
    \int_{A} f(x,y;\sigma)dy & =&  \int_{A\cap B(x;R)} f(x,y;\sigma) dy +  \int_{A \setminus B(x;R)} f(x,y;\sigma) dy \nonumber \\
    &\leq&  \int_{A\cap B(x;R)} f(x,y;\sigma) dy +  |A \setminus B(x;R)| \max_{y\in A \setminus B(x;R)} f(x,y;\sigma) \nonumber\\
    &\leq& \int_{A\cap B(x;R)} f(x,y;\sigma) dy +  |A \setminus B(x;R)| e^{-\frac{R^2}{2\sigma^2}},\label{eq:lhs} \\
\int_{B(x;R)} f(x,y;\sigma)dy & =&  \int_{A\cap B(x;R)} f(x,y;\sigma) dy +  \int_{B(x;R) \setminus A} f(x,y;\sigma) dy  \nonumber\\
&\geq& \int_{A\cap B(x;R)} f(x,y;\sigma) dy +  |B(x;R) \setminus A| \min_{y\in B(x;R) \setminus A} f(x,y;\sigma) \nonumber\\
&\geq& \int_{A\cap B(x;R)} f(x,y;\sigma) dy +  |B(x;R) \setminus A| e^{-\frac{R^2}{2\sigma^2}}.\label{eq:rhs}
\end{eqnarray}
Since $|B(x;R)\setminus A| = |A \setminus B(x;R)|$ as a result of $|B(x;R)|=|A| = V$, we conclude with~\eqref{eq:ball} by combining~\eqref{eq:lhs} with~\eqref{eq:rhs}.
\end{proof}

Next, we obtain the following key lemma, which shows that we can ``restart'' and ``forget'' the history of $X_{n-1}, \ldots, X_0$ once the $n$-th iterate $X_n \in \Omega_n^c$.

\begin{lemma}\label{lem:cond_history1}
Consider random variables $\{X_n\}$ generated by the stochastic algorithm~\eqref{eq:diffdisc1}. For any $0\leq n < m <\infty$, and  arbitrary measurable sets $S_{n},\ldots,S_{m-1}, S_{m} \subseteq \mathfrak{X}$ where $S_{m-1}\subseteq \Omega_{m-1}^c$, the conditional probability 
\begin{equation}
    \bbP\left(X_{m} \in S_{m} |   X_{m-1} \in S_{m-1}, \cdots, X_{n} \in S_{n}\right) = \bbP\left(X_{m} \in S_{m} |  X_{m-1} \in S_{m-1}\right) = \frac{|S_m|}{|\fX|}\,.
\end{equation}
\end{lemma}
\begin{proof}
    The main algorithm in~\eqref{eq:diffdisc1} yields a discrete-time Markov chain on a measurable state space $\mathfrak{X}$. The joint probability is given by
    \begin{eqnarray}
&& \bbP\left( X_{n} \in S_{n},\, X_{n + 1} \in S_{n+1},\, \cdots, \, X_{m} \in S_{m}\right) \\
&=& \int_{S_n}\int_{S_{n+1}} \cdots \int_{S_{m-1}}\int_{S_{m}} p_{m-1}(x_{m-1}, dx_m)\,  p_{m-2}(x_{m-2}, dx_{m-1})\, \cdots \, p_n(x_{n}, dx_{n+1}) \mu_n(dx_n)\,, \nonumber
\end{eqnarray}
where $p_n$ is the Markov kernel from $X_n$ to $X_{n+1}$ and $\mu_n$ is the associated probability measure for $X_n$. Based on the particular bi-level variance in~\eqref{eq:diffdisc1}, the Markov kernel $p_n$ has an analytic form
\begin{equation}\label{eq:transition_kernel}
    p_n(x_{n}, dx_{n+1})   = \left( \mathbbm{1}_{x_n \in \Omega_n}  
     k_n(x_{n+1}, x_{n} ) + 
    \mathbbm{1}_{x_n \not \in \Omega_n}  \frac{1}{|\mathfrak X|}
     \right) dx_{n+1}\,,\, \forall n\,,
\end{equation}
where $k_n$ is defined in~\eqref{eq:kernel}, i.e.,
\begin{equation*}
    k_n(x_{n+1}, x_{n} )  = \left(\sqrt{2\pi } \sigma_n\right)^{-d}    \exp \bigg(-\frac{|x_{n+1}-x_n^+ |^2}{2\sigma^2_n}\bigg),\quad x_n^+ =  x_n -\eta_n G(x_n).
\end{equation*}
In particular, if $S_{m-1} \subseteq \Omega_{m-1}^c$, which implies that $x_{m-1} \not\in \Omega_{m-1}$, we have that
\[
p_{m-1}(x_{m-1}, dx_m) =  |\mathfrak X|^{-1} dx_m.
\]
Therefore, the conditional probability 
\begin{eqnarray}
&&\bbP\left(X_{m} \in S_{m} \,|\, X_{n} \in S_{n},\, X_{n + 1} \in S_{n+1},\, \cdots, \, X_{m-1} \in S_{m-1} 
\right)  \nonumber \\
&=& \frac{ \bbP \left( X_{n} \in S_{n},\, X_{n + 1} \in S_{n+1},\, \cdots, \, X_{m-1} \in S_{m-1},\, X_{m} \in S_{m} 
\right) }{\bbP\left( X_{n} \in S_{n},\, X_{n + 1} \in S_{n+1},\, \cdots, \, X_{m-1} \in S_{m-1}
\right) }  \label{eq:cond_expand} \\
&=& \frac{\int_{S_n}\int_{S_{n+1}} \cdots \int_{S_{m-1}} \left(\int_{S_{m}} 
|\mathfrak X|^{-1}  dx_m \right)\,  p_{m-2}(x_{m-2}, dx_{m-1})\, \cdots \, p_n(x_{n}, dx_{n+1}) \mu_n(dx_n)}{\int_{S_n}\int_{S_{n+1}} \cdots \int_{S_{m-1}} p_{m-2}(x_{m-2}, dx_{m-1})\, \cdots \, p_n(x_{n}, dx_{n+1}) \mu_n(dx_n)} \nonumber \\
&=& \int_{S_{m}} 
|\mathfrak X|^{-1}  dx_m   =  \bbP\left(X_{m} \in S_{m} |  X_{m-1} \in S_{m-1}\right), \nonumber
\end{eqnarray}
which finishes the proof. 
\end{proof}

Let us define a key ratio that will be used in the rest of the section,
\begin{equation} \label{eq:s_n}
    s_n = c_0\sqrt[d]{|\Omega_{n}|}/{\sigma}_{n}\,.
\end{equation}
We next provide a lower bound for the probability of ``jumping out'' of the lower-variance region.

\begin{lemma}\label{lem:cond_history2}
Consider random variables $\{X_n\}$ generated by the stochastic algorithm~\eqref{eq:diffdisc1}. For any $0\leq n < m <\infty$, and arbitrary  measurable sets $S_{n},\ldots, S_{m-1} \subseteq \mathfrak{X}$ where $S_{m-1} \subseteq \Omega_{m-1}$, the conditional probability satisfies
\begin{equation}\label{eq:pn_lower_bd}
    \bbP\left(X_{m} \in \Omega_m^c \,|\,X_{m-1} \in S_{m-1}, \, \cdots\, , X_{n} \in S_{n} \right)   
\geq  c\, s_{m-1}^d e^{-s_{m-1}^2} =: \ell_{m} \,,
\end{equation}
where $s_{m-1}$ is defined in~\eqref{eq:s_n} and the constant $c$ only depends on 
 the dimension $d$.
\end{lemma}

\begin{proof}
If $x_{m-1} \in \Omega_{m-1}$, based on~\eqref{eq:transition_kernel} and~\Cref{lem:ball}, we have that
\[
\int_{S_m} p_{m-1}(x_{m-1}, dx_{m}) = \int_{S_m}  k_{m-1} (x_m,x_{m-1}) dx_m \leq \int_{B(x^+_{m-1}, r)}  k_{m-1} (x_m,x_{m-1}) dx_m,
\]
where the ball $B(x_{m-1}^+, r)$ is centered at $x^+_{m-1} = x_{m-1} - \eta_{m-1} G(x_{m-1})$ with radius $r$ such that $|B(x^+_{m-1}, r)| = |S_m|$. We can further investigate the upper bound:
\begin{eqnarray*}
  \int_{B(x^+_{m-1}, r)}  k_{m-1} (x_m,x_{m-1}) dx_m  &=&   \left(\sqrt{2\pi } \sigma_{m-1}\right)^{-d}   \int_{B(x^+_{m-1}, r)} \exp \bigg(-\frac{|x_{m}-x_{m-1}^+ |^2}{2\sigma^2_{m-1}}\bigg) \, dx_m\\
  &=& \left(\sqrt{2\pi } \sigma_{m-1}\right)^{-d}   \int_{B(0, r)} \exp \bigg(-\frac{|x|^2}{2\sigma^2_{m-1}}\bigg) \, dx\, =: C(|S_m|, \sigma_{m-1}),
\end{eqnarray*}
due to the translation invariance of the isotropic Gaussian kernel.
As a result, the upper bound is independent of the location of $x_{m-1}$ except the fact that  $x_{m-1} \in \Omega_{m-1}$. It only depends on two constants, the volume $|S_m| $ and $\sigma_{m-1}$.

Next, we revisit the expansion in~\eqref{eq:cond_expand} for this scenario:
\begin{eqnarray*}
&&\bbP\left(X_{m} \in S_{m} \,|\, X_{n} \in S_{n},\, \cdots, \, X_{m-1} \in S_{m-1} 
\right)  \nonumber \\
&=& \frac{\int_{S_n}\cdots \int_{S_{m-1}} \left(\int_{S_{m}}    p_{m-1}(x_{m-1}, dx_{m})\right)
\,  p_{m-2}(x_{m-2}, dx_{m-1})\, \cdots \mu_n(dx_n)}{\int_{S_n}\cdots \int_{S_{m-1}} p_{m-2}(x_{m-2}, dx_{m-1})\, \cdots \, \mu_n(dx_n)}  \\  
&\leq & \frac{\int_{S_n}\cdots \int_{S_{m-1}} C(|S_m|, \sigma_{m-1})
\,  p_{m-2}(x_{m-2}, dx_{m-1})\, \cdots \mu_n(dx_n)}{\int_{S_n}\cdots \int_{S_{m-1}} p_{m-2}(x_{m-2}, dx_{m-1})\, \cdots \, \mu_n(dx_n)} \\
&=& C(|S_m|, \sigma_{m-1}).
\end{eqnarray*}
In particular, if $S_m = \Omega_m$, we have the following lower bound:
\begin{eqnarray*}
&&\bbP\left(X_{m} \in \Omega_m^c \,|\, X_{n} \in S_{n},\, \cdots, \, X_{m-1} \in S_{m-1} \right)\\
&=& 1 - \bbP\left(X_{m} \in \Omega_m \,|\, X_{n} \in S_{n},\, \cdots, \, X_{m-1} \in S_{m-1} \right) \\
&\geq & 1 - C(|\Omega_m|, \sigma_{m-1}) \\
&=& \frac{1}{(\sqrt{2\pi } \sigma_{m-1})^{d} }    \int_{|x|>c_0 \sqrt[d]{|\Omega_{m}|}} \, \exp \bigg(-\frac{|x|^2}{2\sigma^2_{m-1}}\bigg) \, dx \\
&\geq & \frac{1}{(\sqrt{2\pi } \sigma_{m-1})^{d} }  \int_{\sqrt{2}r_{m-1} \geq |x|>r_{m-1}} \exp \bigg(-\frac{|x|^2}{2\sigma^2_{m-1}}\bigg) \, dx\quad \left( \text{where }r_{m-1} = c_0 \sqrt[d]{|\Omega_{m-1}|}\right)\\
&\geq& \frac{(2^{d/2} -1) r_{m-1}^d c_0^{-d}}{\left(\sqrt{2\pi } \sigma_{m-1}\right)^{d}} \exp \bigg(-\frac{r_{m-1}^2}{\sigma^2_{m-1}}\bigg) = \beta_1 s_{m-1}^d e^{-s_{m-1}^2}\,,
\end{eqnarray*}
where $s_n$ is given in~\eqref{eq:s_n} and $\beta_1 = \frac{2^{d/2}-1}{ (\sqrt{2\pi} c_0)^{d}}$.
\end{proof}

We are ready to present the main lemma showing that for any given set with a positive Lebesgue measure, at least one of the history iterates can visit the set almost surely.
\begin{lemma}\label{thm:prop1}
Under the assumptions in~\ref{itm:B2} and~\ref{itm:B3}, for any closed subset $\Omega\subseteq \mathfrak{X}$ where $|\Omega|>0$, we have
\begin{equation}\label{eq:prop_1_goal}
   \mathbb{P}\left(\bigcup_{n=0}^\infty \{ X_{n}\in \Omega \}\right)= 1\,.
   \end{equation}
\end{lemma}
\begin{proof}
Recall that we denote the event $X_n \in \Omega$ by $B_n$. We will prove the equivalence of~\eqref{eq:prop_1_goal}:
\begin{equation*}
    \mathbb{P}\left(\bigcap_{n=0}^\infty \{ X_{n} \not\in \Omega \}\right) = \lim_{N\to\infty} \bbP(B_{N+1}^c\cap B_{N}^c\cap \cdots \cap B_0^c)= 0\,.
\end{equation*}
Without loss of generality, let us assume that $N$ is odd. We first observe that
\begin{equation*}
    \bbP(B_{N+1}^c\cap B_{N}^c\cap \cdots \cap B_0^c) \le \bbP(B_{N+1}^c\cap B_{N-1}^c\cap B_{N-3}^c \cdots \cap B_0^c)\,.
\end{equation*}
Next, we define the event
\[
    C_{N+1}:=B_{N+1}^c\cap B_{N-1}^c\cap B_{N-3}^c \cdots \cap B_0^c\,.
\]
It is sufficient to show that $\lim_{N\rightarrow \infty}\bbP(C_{N+1})=0$ to finish the proof.

Based on its definition, we have
\begin{multline*}
    \bbP(C_{N+1})=\bbP(B_{N+1}^c\cap C_{N-1})=\bbP(B_{N+1}^c|C_{N-1})\bbP(C_{N-1}) \\
    =\left(1-\bbP(B_{N+1}|C_{N-1}) \right)\bbP(C_{N-1})=\bbP(C_{N-1})-\bbP(B_{N+1}\cap C_{N-1})\,.
\end{multline*}
This gives us the relation
\begin{equation}\label{eq:C_1}
\bbP(C_{N+1})=\bbP(C_{N-1})-\bbP(B_{N+1}\cap C_{N-1})\,.
\end{equation}
The next step is to bound the term $\bbP(B_{N+1}\cap C_{N-1})$ from below. To do this, we first observe that
\begin{equation}\label{EQ:intermediate}
    \bbP(B_{N+1}\cap C_{N-1})=\bbP\left(B_{N+1}\cap (A_N\cup A_N^c)\cap(A_{N-1}\cup A_{N-1}^c)\cap C_{N-1} \right)\,. %
\end{equation}
Utilizing the relations
\begin{eqnarray*}
    \alpha \cap (\beta\cup \beta^c)\cap(\gamma\cup \gamma^c)\cap \delta&=&\Big[(\alpha\cap \beta)\cup(\alpha\cap \beta^c)\Big] \cap \Big[(\gamma\cap \delta)\cup (\gamma^c\cap \delta)\Big]\,,\\
    \left(\alpha \cup \beta) \cap (\gamma \cup \delta \right) &=& (\alpha \cap \gamma) \cup (\beta \cap \gamma )\cup (\alpha \cap \delta) \cup (\beta \cap \delta)\,,
\end{eqnarray*}
we can rewrite~\eqref{EQ:intermediate} into
\begin{align}\label{EQ:intermediate2}
\bbP(B_{N+1}\cap C_{N-1})&=\bbP\big((\cA\cap \cC)\cup(\cB\cap \cC)\cup(\cA\cap \cD)\cup(\cB\cap \cD)\big)  \nonumber \\ 
&=\bbP ((\cA\cap \cC)\cup(\cB\cap \cC))+\bbP((\cA\cap \cD)\cup(\cB\cap \cD))-  \nonumber\\
&\quad \bbP \left(  \left((\cA\cap \cC)\cup (\cB\cap \cC) \right)\cap \left ((\cA\cap \cD)\cup(\cB\cap \cD) \right) \right)\,,
\end{align}
where
\[
    \cA:=B_{N+1}\cap A_N,\quad \cB:=B_{N+1}\cap A_N^c,\quad \cC:=A_{N-1}\cap C_{N-1},\quad \cD=A_{N-1}^c\cap C_{N-1}\,.
\]
From the definition, we see that $\cA\perp \cB$ and $\cC\perp \cD$. This immediately gives us that the last term of~\eqref{EQ:intermediate2} is zero. Moreover, \eqref{EQ:intermediate2} can be simplified to
\[
\bbP(B_{N+1}\cap C_{N-1})=\bbP(\cA\cap \cC)+\bbP(\cB\cap \cC)+\bbP(\cA\cap \cD)+\bbP(\cB\cap \cD)\,,
\]
since $ \bbP(\cA\cap \cC\cap \cB\cap \cC) = \bbP(\cA\cap \cD\cap \cB\cap \cD)=0$.

Meanwhile, using the fact that $\bbP(\cA\cap \cC)\ge 0$ and $\bbP(\cA\cap \cD)\ge 0$, we have
\begin{multline}
    \bbP(B_{N+1}\cap C_{N+1})\ge \bbP(\cB\cap \cC)+\bbP(\cB\cap \cD)\\
    =\bbP(B_{N+1}\cap A_N^c\cap A_{N-1}\cap C_{N-1})+\bbP(B_{N+1}\cap A_N^c\cap A_{N-1}^c\cap C_{N-1})\,.
\end{multline}
We can express the joint distribution using the product of several conditional distributions:
\begin{multline}
    \bbP(B_{N+1}\cap A_{N}^c\cap A_{N-1}\cap C_{N-1})=\\
    \bbP(B_{N+1}| A_{N}^c\cap A_{N-1}\cap C_{N-1})   \bbP(A_{N}^c| A_{N-1}\cap C_{N-1})\bbP(A_{N-1}| C_{N-1})\bbP(C_{N-1})\,.   
\end{multline}
We define $t_{N-1} := \bbP(A_{N-1}^c|C_{N-1})$, and thus $\bbP(A_{N-1}| C_{N-1}) = 1-t_{N-1}$.  The first term $\bbP(B_{N+1}| A_{N}^c\cap A_{N-1})$ is a constant thanks to~\Cref{lem:cond_history1} as
\[
\bbP(B_{N+1}| A_{N}^c\cap A_{N-1}) = \bbP(B_{N+1}| A_{N}^c) = |\Omega|/|\fX| =: C_\Omega. 
\]
Based on~\Cref{lem:cond_history2}, the second term on the right-hand side $\bbP(A_{N}^c| A_{N-1}\cap C_{N-1})$   is lower bounded by $\ell_{N}$.  Similarly,
\begin{multline}
    \bbP(B_{N+1}\cap A_{N}^c\cap A^c_{N-1}\cap C_{N-1})=\\
    \bbP(B_{N+1}| A_{N}^c\cap A^c_{N-1}\cap C_{N-1})   \bbP(A_{N}^c| A^c_{N-1}\cap C_{N-1})\bbP(A^c_{N-1}| C_{N-1})\bbP(C_{N-1})\,.   
\end{multline}
The first term on the right-hand sided equals to $C_\Omega$, the second term equals to $|\Omega_{N}|/|\fX|=: \delta_{N} $, and the third term is $t_{N-1}$. 

Note that $\delta_{N}  = \mathcal{O}(N^{-\alpha})$ where $0 < \alpha < 1$, as a result of~\ref{itm:B3}. With Assumption~\ref{itm:B2}, $\ell_{N} = c \frac{(\log (N-1) )^{d/2} }{N-1}$ where the constant $c$ only depends on $d$. Clearly, $\ell_{N}$ decays to zero faster than $\delta_{N}$ as $N\rightarrow \infty$. Consider an odd integer $M\in \mathbb{N}$ such that for all $N \geq M$, $\ell_N < \delta_{N}$. 
Together with~\eqref{eq:C_1}, we have
\begin{eqnarray*}
    \bbP(C_{N+1}) &\leq& \bbP(C_{N-1})-\Big[(C_\Omega \cdot \delta_N \cdot t_{N-1})+(C_\Omega \cdot \ell_{N}(1-t_{N-1}) \Big]\bbP(C_{N-1})\\
    &\leq & (1- C_\Omega \ell_N )\bbP(C_{N-1}) \\
    &\leq & (1- C_\Omega \ell_N ) (1- C_\Omega \ell_{N-2} ) \cdots (1- C_\Omega \ell_M)\,.
\end{eqnarray*}
Since $\ell_{k}$ decays to zero slower than $1/k$, we also have
\[
    \prod_{k=M, odd}^\infty \left(1- C_\Omega \ell_k \right) =0\,,
\]
which concludes the proof.
\end{proof}

\subsection{Property Two: Jumping-In Easier Than Jumping-Out}\label{sec:prop2}
In~\Cref{sec:prop1}, we have shown that the sequence of iterates $\{X_n\}$ can almost surely visit any subset of $\mathfrak{X}$ with a positive Lebesgue measure at least once. %
Next, we will show that under the assumptions stated in~\Cref{sec:main}, the iterate $X_n$ is more likely to jump from $\Omega_n^c$ into the set $\Omega_{n+1}$ than to jump from $\Omega_n$ to $\Omega_{n+1}^c$. That is,
\begin{equation} \label{eq:jumpout_vs_jumpin}
\lim\limits_{n\rightarrow\infty} \frac{\Prob ( X_{n+1}\in\Omega_{n+1}^c|X_{n}\in\myin )}{\Prob ( X_{n+1}\in \Omega_{n+1}|X_n\in\myout )} = \lim\limits_{n\rightarrow\infty} \frac{p_n(\Omega_{n+1}^c)}{q_n(\Omega_{n+1})} = 0.
\end{equation}
This is a necessary property for the SGD algorithm in~\eqref{eq:diffdisc1} to guarantee global convergence. While $q_n(\Omega_{n+1}) = \frac{|\Omega_{n+1}|}{|\mathfrak{X}|}$ is given explicitly, we will use an upper bound for $p_n(\Omega_{n+1}^c)$ in the following lemmas. The technique here is different from those in~\Cref{sec:prop1} where a lower bound of $p_n(\Omega_{n+1}^c)$ was used.

Consider $x$ as a realization of the random variable $X_n$. Recall $x^+ = x - \eta_n G(x)$ as used in~\eqref{eq:kernel}. First, we justify that for large $n$, $x^+ \in \Omega_{n+1}$ if $x\in\Omega_n$ as a realization of $X_n$.

\begin{lemma}\label{lem:xn+}
Under the assumptions in~\ref{itm:A1}-~\ref{itm:A2}, ~\ref{itm:B1} and~\ref{itm:B3}, for any $\gamma > \gamma^* =(2\eta b_2 - \eta^2 b_2^2)^{-1}$, we choose $\epsilon >0$ sufficiently small so that $B(x^*;\epsilon)\subseteq \Omega_{sc}$ and
\begin{equation}\label{eq:bgamma}
b_\gamma  \nabla^2 f(x^*) \leq \nabla^2 f(z)  \leq  b_\gamma ^{-1} \nabla^2 f(x^*), \qquad b_\gamma= \left(1 - \frac{b_1}{\gamma b_2} \right)^{1/2} \leq 1,
\end{equation}
for all $z\in B(x^*;\epsilon)$, where the inequalities hold entry-wise. Then, $\exists N \in \N$ such that $\forall n \geq N$, $\Omega_{n} \subset B(x^*;\epsilon)$. Moreover, if $x\in \Omega_n$ and $n> N$, then $x^+  = x - \eta_n G(x)\in \Omega_{n+1}$.
\end{lemma}
\begin{proof}
Consider our global minimizer $x^*$ and its $\epsilon$-neighborhood $B(x^*;\epsilon)$ with the given $\epsilon$. By construction, the sequence $\{f_n\}$ contains the scalar cutoff values in~\eqref{eq:diffdisc1}. By~\ref{itm:B3}, the volume of the sublevel set $\Omega_n$ determined by $f_n$ decays to $0$ as $n$ increases, so  we have $\lim_{n\rightarrow\infty} {f_n} = f(x^*) = \fkf^*$. Thus, $\exists N_1 \in \mathbb{N}$ such that for any $n\geq N_1$, $\Omega_n  \subset  B(x^*;\epsilon) \subseteq \Omega_{sc}$. Next, we consider $x\in \Omega_n$ for $n > N_1$ such that $\Omega_{n} \subseteq \Omega_{n-1} \subset  B(x^*;\epsilon)$.

Based on Assumptions \ref{itm:A1}-\ref{itm:A2}, the following holds for any $y_1,y_2\in \Omega_{sc}$.
\begin{enumerate}[label = \textbf{C\arabic*}]
    \item \label{itm:C1} \(
        G(y_1)^T(y_2-y_1) + b_1/2|y_2-y_1|_2^2 \leq f(y_2) - f(y_1) \leq   G(y_1)^T(y_2-y_1) + b_2/2|y_2-y_1|_2^2.\)
    \item \label{itm:C2}  \( \sqrt{2b_1 (f(y_1)-\fkf^*) } \leq |G(y_1)| \leq   \sqrt{2b_2 (f(y_1)-\fkf^*) },\,\text{with\ } \fkf^* = f(x^*). \)
\end{enumerate}
Here $|\cdot|_2$ denotes the Euclidean ($\ell^2$) norm of the vector. We  plug $y_1= x$ and $y_2= x^+ = x - \eta_n G(x)$ into \ref{itm:C1} and \ref{itm:C2}. Since $0 < 
\eta_n = \eta < 2/b_2$ (by Assumption~\ref{itm:B1}), we have $\eta - b_2\eta^2/2 > 0$ and
\begin{eqnarray}
    f(x^+) - \fkf^*  &\leq&   f(x) -\fkf^* - \left(\eta - b_2\eta^2/2 \right) |G(x)|^2
 \nonumber \\
&\leq &  (f(x)- \fkf^*) \left(1 -2\eta b_1 + \eta^2 b_1 b_2\right)  =  (f(x)-\fkf^*) \left(1 - \frac{b_1}{\gamma^* b_2}\right), \label{eq:J_ratio}
\end{eqnarray}
where $\gamma^* =(2\eta b_2 - \eta^2 b_2^2 )^{-1} = (1-(\eta b_2 - 1)^2)^{-1} \geq 1$ as a result of $0 < \eta < 2/b_2$. If $x\in\Omega_{n+1}\subseteq \Omega_n$, \eqref{eq:J_ratio} also implies $x^+ \in \Omega_{n+1}$ since $ f(x^+) < f(x) \leq f_{n+1}$. Therefore, we focus on the case where $x \in \Omega_n \setminus \Omega_{n+1}$. The set difference is non-empty since $|\Omega_{n+1}| < |\Omega_n|$ by Assumption~\ref{itm:B3}. 

From the Taylor expansion of $f$ around $x^*$ and the multi-dimensional mean-value theorem, $\forall y\in B(x^*;\epsilon)$, we have 
\[
 f(y) = f(x^*) + \nabla f(x^*)^\top (y-x^*) +  \frac{1}{2} (y-x^*)^\top \nabla^2 f(z(y)) (y-x^*),\quad z(y) = \tau x^* + (1-\tau) y,
\]
for some $0\leq \tau \leq 1$. Since $y \in B(x^*;\epsilon)$, we also have that $z(y) \in B(x^*;\epsilon)$. Moreover, as a result of $\nabla f(x^*) = 0$ and $f(x^*) = \fkf^*$, together with~\eqref{eq:bgamma}, $\forall y \in  B(x^*;\epsilon)$, we have
\begin{equation}\label{eq:f upperlower bound}
\frac{b_\gamma}{2} (y-x^*)^\top  \nabla^2 f(x^*) (y-x^*) \leq  f(y)-\fkf^* \leq \frac{1}{2b_\gamma} (y-x^*)^\top \nabla^2 f(x^*) (y-x^*).
\end{equation}

Next, we restrict $y$ to $\Omega_{n-1} \setminus \Omega_{n}$, $\forall n > N_1$, a subset of 
$B(x^*;\epsilon)$. The set difference is non-empty since $|\Omega_n| < |\Omega_{n-1}|$ by~\ref{itm:B3}. Using the fact that $f_{n}< f(y) \leq f_{n-1}$ and~\eqref{eq:f upperlower bound}, we obtain that
\begin{equation}\label{eq:Omega_n_area_1}
\Omega_{n-1}  \setminus \Omega_{n} \subseteq \overline{E}_{n-1} \setminus \underline{E}_{n},
\end{equation}
where the two sets $\overline{E}_{n-1}$ and $\underline{E}_{n}$ are two nested ellipsoids defined by
\begin{eqnarray*}
 \overline{E}_{n} &:=& \big\{x | (x-x^*)^\top  \nabla^2  f(x^*) (x-x^*) \leq 2b_\gamma^{-1} \left( f_{n} - \fkf^*\right)\big\} \label{eq:overline En},\\
 \underline{E}_{n} &:=& \big\{x | (x-x^*)^\top  \nabla^2  f(x^*) (x-x^*) \leq 2b_\gamma \left( {f_{n}} - \fkf^*\right)\big\}.\label{eq:underline En}
\end{eqnarray*}

We then have the important set relation for any $n > N_1$,
\begin{equation}\label{eq:Omega_n_area_2}
\underline{E}_{n} \subseteq \Omega_{n} \subseteq \Omega_{n-1} \subseteq \overline{E}_{n-1},\quad \underline{E}_{n+1} \subseteq \Omega_{n+1} \subseteq \Omega_{n} \subseteq \overline{E}_{n}.
\end{equation}
Combining the two set relations, we have
\begin{equation}\label{eq:Omega_n_area_3}
 \underline{E}_{n} \subseteq \Omega_{n}   \subseteq \overline{E}_{n},\quad n > N_1.
\end{equation}
In other words, $\Omega_{n}$ is nested between two ellipsoids centered at $x^*$ for $n > N_1$. Its volume is thus bounded in between by $ |\underline{E}_{n}|$ and $|\overline{E}_{n}|$. Define $c_3= (\sqrt{2}c_0^{-1})^d /\sqrt{\det (\nabla^2 f(x^*))}$, we have
\begin{eqnarray}\label{eq:omega_n_volume}
|\underline{E}_{n}| = c_3 b_\gamma^{d/2} (f_{n} - \fkf^*)^{d/2} \leq &   |\Omega_{n}|  &\leq    c_3 b_\gamma^{-d/2} (f_{n} - \fkf^*)^{d/2} = |\overline{E}_{n}|,\\
b_\gamma c_3^{-2/d}  |\Omega_{n}|^{2/d} \leq &  \, f_{n} - \fkf^*\, &\leq   b_\gamma^{-1}  c_3^{-2/d}  |\Omega_{n}|^{2/d}, \label{eq:cutoff_bound}
\end{eqnarray}
for $n > N_1$. Based on Assumption~\ref{itm:B3}, for $n\geq 1$, $|\Omega_{n+1}| = |\Omega_0| n^{-\alpha}$ for some $\alpha \in (0,1)$. Together with \eqref{eq:cutoff_bound}, we have the following for $n >  \max (2,N_1)$,
\begin{equation}\label{eq:bn}
b_n  \leq \frac{{f_{n+1}} - \fkf^* }{{f_n} - \fkf^*} \leq 1,\quad \text{where } b_n = b_\gamma^{2} \left(\frac{n}{n-1}\right)^{-\frac{2\alpha}{d}} = \left(1 - \frac{b_1}{\gamma b_2} \right) \left(1-  \frac{1}{n}\right)^{\frac{2\alpha}{d}}.
\end{equation}
On the other hand, since $\gamma > \gamma^* \geq 1$, we know that $b_\gamma > b_{\gamma^*}$ based on~\eqref{eq:bgamma}. Therefore, let $N_2$ be the the  smallest integer that is not smaller than  $\left( 1 - (b_{\gamma^*}/b_\gamma)^{\frac{d}{\alpha}} \right)^{-1}$. Then $\forall n \geq N_2$, we have $b_n \geq  1 - \dfrac{b_1}{\gamma^* b_2}$. 
Together with~\eqref{eq:J_ratio}, we find that
\begin{equation} \label{eq:f_ratio}
\frac{f(x^+)-\fkf^* }{{f_n} - \fkf^* } \leq \frac{f(x^+) - \fkf^*}{f(x) - \fkf^*} \leq  1 - \frac{b_1}{\gamma^* b_2}  \leq   b_n \leq \frac{{f_{n+1}} - \fkf^*}{{f_n} - \fkf^*},\  \forall n\geq N = \max(N_1,N_2).
\end{equation}
Since \eqref{eq:f_ratio} also implies $f(x^+) \leq {f_{n+1}}$, we conclude that $x^+ \in \Omega_{n+1}$ provided $x\in \myin$. 
\end{proof}

Next, we provide a lower bound for $    \frac{d_n(x)}{c_0\sqrt[d]{|\Omega_n|}}$, where  
\begin{equation} \label{eq:d_n def}
    d_n(x) =  \inf\limits_{y\in \Omega_{n+1}^c} |y- x^+|
\end{equation}
for a given $x \in\Omega_n$ and $x^+ = x - \eta_n G(x)$.

\begin{lemma}\label{lem:d_n_vs_Omega_n}
Let Assumption~\ref{itm:B2} and all conditions in~\Cref{lem:xn+} hold. There exists $N_3\in\N$ such that $\forall n\geq N_3$, we have $c^* \leq \frac{d_n(x)}{c_0\sqrt[d]{|\Omega_n|}} \leq 1$ for all $x\in\Omega_n$, where $c^* = \frac{1}{4\gamma^*} \left(\frac{b_1}{b_2}\right)^{\frac{3}{2}} $ and $\gamma^* = (2\eta b_2 - \eta^2 b_2^2)^{-1}$.
\end{lemma}
\begin{proof}
Let $N$, $\gamma$ be those given in the conditions of~\Cref{lem:xn+}, whose conclusions hold here by the assumptions. We consider the case that $x\in \Omega_n$ where $n>N$. Here, $x$ is a possible realization of the random variable $X_n$. Thus, we have $x^+ \in \Omega_{n+1}$ and $\Omega_{n+1}\subseteq \Omega_{sc}$.  We also have that
\begin{equation*}%
d_n(x) = \inf_{y\in \Omega_{n+1}^c} |y- x^+| = \inf\limits_{y\in \partial\Omega_{n+1}^c} |y- x^+| = \inf\limits_{y\in \partial\Omega_{n+1}} |y- x^+|.
\end{equation*}
See~\Cref{fig:omega_n} for an illustration of $x^+$, $d_n(x)$ and $\Omega_{n+1}$. Consider a ball $B(x^+;d_n(x))$ centered at $x^+$ with a radius $d_n(x)$. Since $B(x^+;d_n(x)) \subseteq \Omega_{n+1} \subseteq \Omega_n$, $|B(x^+;d_n(x))| \leq  |\Omega_n|$, or equivalently, $d_n(x)/R_n =d_n(x) /(c_0 \sqrt[d]{|\Omega_n|}) \leq 1$, where $R_n = c_0\sqrt[d]{|\Omega_n|}$. %
Thus, we have the upper bound for $d_n(x)$.

Combining \ref{itm:C1} and \ref{itm:C2} in~\Cref{lem:xn+}, we find that
\begin{eqnarray}
f(y_2) - f(y_1)  &\leq & \sqrt{2b_2 (f(y_1) - \fkf^*)} |y_2-y_1| + b_2/2|y_2-y_1|_2^2 \\
&=& \big(\sqrt{b_2/2} |y_2-y_1|  + \sqrt{f(y_1) - \fkf^*}  \big)^2 - (f(y_1)-\fkf^*),\label{eq:important_ineq}
\end{eqnarray}
which implies that for any $y_1,y_2\in\Omega_{sc}$, we have
\begin{equation}
    |y_2-y_1| \geq \sqrt{2/b_2}\big(\sqrt{f(y_2) - \fkf^*} - \sqrt{f(y_1) - \fkf^*} \big).\label{eq:important_ineq1}
\end{equation}
By plugging $y_1 = x^+$ and  any $y_2\in\partial\Omega_{n+1}$ into~\eqref{eq:important_ineq1} (note that $y_1,y_2\in\Omega_{sc}$), we have that
\begin{eqnarray}
   d_n(x) & \geq & \sqrt{\frac{2}{b_2}} \inf\limits_{y_2\in \partial\Omega_{n+1}} \left(\sqrt{f(y_2)-\fkf^*} - \sqrt{f(x^+) - \fkf^*} \right) \nonumber \\
   &=& \sqrt{\frac{2}{b_2}}  \left( \sqrt{{f}_{n+1} - \fkf^*} - \sqrt{f(x^+)-\fkf^*} \right).\label{eq:lower_bd_dn}
\end{eqnarray}

Combining~\eqref{eq:lower_bd_dn} with the right-hand side of~\eqref{eq:omega_n_volume} in~\Cref{lem:xn+}, we have
\begin{eqnarray}
   \frac{d_n(x)}{c_0\sqrt[d]{|\Omega_n|}} &\geq&   \frac{\sqrt{2/b_2}  \left(\sqrt{{f}_{n+1}- \fkf^*} - \sqrt{f(x^+) - \fkf^*} \right)}{c_0 b_\gamma^{- \frac{1}{2}} (c_3)^{\frac{1}{d}}  \sqrt{ {f_n} - \fkf^* } } \nonumber \\
   &\geq&  \frac{\sqrt{2/b_2} \left( 1-\frac{b_1}{\gamma b_2}\right)^{\frac{1}{4}} }{\sqrt{2/b_1}} \left(\sqrt{b_n } - \sqrt{\frac{f(x^+) - \fkf^*}{ f_n - \fkf^*}}\right) \nonumber  \\
   & > &  \sqrt{\frac{b_1}{b_2}} \sqrt{1-\frac{b_1}{\gamma b_2}} \left( \sqrt{1 - \frac{b_1}{\gamma b_2} } \left(1 - \frac{1}{n}\right)^{\frac{\alpha}{d}}- \sqrt{1-\frac{b_1}{\gamma^* b_2}}  \right). \label{eq:c*} 
\end{eqnarray}
Here, we have applied~\eqref{eq:f_ratio} in~\Cref{lem:xn+} and the fact that $c_3 = (\sqrt{2}c_0^{-1})^d /\sqrt{\det (\nabla^2 f(x^*))} \leq (\sqrt{2/b_1}c_0^{-1} )^d$ by~\ref{itm:A1}-\ref{itm:A2}. Note that the last term of~\eqref{eq:c*} depends on both $\gamma$ and $n$, but it no longer depends on the position of $x\in\Omega_n$. Moreover,
\begin{equation*}
       \lim_{\gamma\rightarrow \infty}   \lim_{n\rightarrow \infty}   \sqrt{1-\frac{b_1}{\gamma b_2}} \left( \sqrt{1 - \frac{b_1}{\gamma b_2} } \left(1 - \frac{1}{n}\right)^{\frac{\alpha}{d}}- \sqrt{1-\frac{b_1}{\gamma^* b_2}}  \right)  =  1 -\sqrt{1 - \frac{b_1}{\gamma^* b_2}} .
\end{equation*}

If we set $\gamma = 4\gamma^*$, let $N$ to be the corresponding integer in~\Cref{lem:xn+} for this $\gamma$, and choose $N_3>N$ such that $b_n \geq  1 - \frac{b_1}{2\gamma^* b_2}$ (since $\lim\limits_{n\rightarrow \infty} b_n = 1 - \frac{b_1}{\gamma b_2} = 1 - \frac{b_1}{ 4\gamma^* b_2} $), then $\forall n \geq N_3$, \eqref{eq:c*} leads to
\begin{equation}
   \frac{d_n(x)}{c_0\sqrt[d]{|\Omega_n|}} \geq  \sqrt{\frac{b_1}{b_2}} \sqrt{1-\frac{b_1}{4\gamma^* b_2}} \left( \sqrt{1 - \frac{b_1}{2\gamma^* b_2} } - \sqrt{1-\frac{b_1}{\gamma^* b_2}}  \right) 
   \geq  \frac{1}{4\gamma^*} \left( \frac{b_1}{b_2}\right)^{\frac{3}{2}} := c^*.\label{eq:c*_2} 
\end{equation}
\end{proof}
\begin{remark}
In the proof above, we applied~\Cref{lem:xn+} by setting $\gamma = 4\gamma^*$. However, $\gamma$ can be chosen arbitrarily large, and the lower bound $c^*$ can consequently be made bigger by a constant factor. As $\gamma$ increases, the positive integer $N$ in both~\Cref{lem:xn+} and~\Cref{lem:d_n_vs_Omega_n} must correspondingly increase.
\end{remark}

\begin{remark}
So far, we have used Assumption~\ref{itm:B3} to simplify the proof. However, the assumption can be relaxed to, for example, $|\Omega_n| = \mathcal{O} ( n^{-\alpha})$ such that $C_1 n^{-\alpha} \leq |\Omega_n| \leq C_2 n^{-\alpha}$ for some constants $C_1, C_2>0$ when $n$ is large enough. In such cases, \eqref{eq:c*_2} still holds but with a different constant $c^*$.
\end{remark}

The proof of~\Cref{lem:d_n_vs_Omega_n} relies on Assumption~\ref{itm:B1} that the step size $\eta_n$ is a constant as $n$ goes to infinity. We choose $\eta_n = \eta < 2/b_2$ as required for the convergence of the gradient descent algorithm applied to a convex optimization problem. Note that~\ref{itm:B1} requires the knowledge or estimation of $b_2$, also known as the Lipschitz constant of the gradient function $G(x)$ over $\Omega_{sc}$ (see Assumption~\ref{itm:A1}).

If we only aim for convergence but not necessarily an algebraic order of convergence, we can relax Assumption~\ref{itm:B1} by assuming $\lim\limits_{n\rightarrow \infty }\eta_n = 0$, which drops the requirement of estimating $b_2$. For example, we may replace \ref{itm:B1}-\ref{itm:B3} by the following set of assumptions.
\begin{enumerate}[label = $\overline{ \textbf{B\arabic*} }$]
\item The step size in~\eqref{eq:diffdisc1} satisfies $\eta_n = \mathcal{O}\left( \frac{1}{\log (\log n )}\right)$.
\item For $n\geq 1$, ${\sigma}_{n} = \dfrac{c_0 \sqrt[d]{|\Omega_n|}}{\sqrt{ \log n}}$. 
\item $|\Omega_n| = \mathcal{O}\left(\frac{1}{\log n}\right)$ based on the properly chosen $\{{f}_n\}$.
\end{enumerate}
Under Assumptions $\overline{\ref{itm:B1}}$-$\overline{\ref{itm:B3}}$,~\eqref{eq:jumpout_vs_jumpin} still holds, but we will no longer have the algebraic-rate estimation for the upper and lower bounds stated in~\Cref{thm:prop2} below. As a result, the algebraic convergence in~\Cref{thm:main} is replaced by a slower logarithmic convergence.

We will see the fact that $d_n(x) /\sqrt[d]{|\Omega_n|}$ is bounded below by a constant independent of both $n$ and $x$ is the key to obtaining an algebraic rate of convergence. Next, we combine results from~\Cref{lem:xn+} and~\Cref{lem:d_n_vs_Omega_n} to show~\eqref{eq:jumpout_vs_jumpin} holds.

\begin{lemma}\label{thm:prop2}
Under the assumptions ~\ref{itm:A1}-\ref{itm:A2} and~\ref{itm:B1}-\ref{itm:B3}, we define $\kappa := |\Omega_0|/|\mathfrak{X}|$. Then we have 
\begin{equation} \label{eq:out_in}
    q_n(\Omega_{n+1}) = \kappa n^{-\alpha},\quad \forall n\geq 1\,.
\end{equation}
Moreover, there exists $N\in \N$ such that
\begin{equation}\label{eq:in_out}
\beta_1 n^{-1}\leq p_n(\Omega_{n+1}^c) \leq \beta_2n^{-\alpha_2}, \quad \forall n \geq N,
\end{equation}
where $\alpha_2 = \frac{(c^*)^2}{4} + \frac{\alpha}{2}$, $\beta_1 =  \frac{2^{d/2}-1}{ (\sqrt{2\pi} c_0)^{d}}$ and $\beta_2 =1$. In particular, if $ 0 < \alpha < \frac{(c^*)^2}{2} $, then \eqref{eq:jumpout_vs_jumpin} holds.
\end{lemma}

\begin{proof}
Based on~\eqref{eq:q_n}, we have $\frac{q_n(\Omega_{n+1})}{|\Omega_{n+1}|}  = \frac{1}{|\mathfrak{X}|}$ so~\eqref{eq:out_in} follows directly from \ref{itm:B3}. To prove that~\eqref{eq:jumpout_vs_jumpin} holds when $ 0 < \alpha < \frac{(c^*)^2}{2} $, we show that  $\lim\limits_{n\rightarrow\infty}\frac{p_n(\Omega_{n+1}^c)}{|\Omega_{n+1}|} = 0$.  Based on~\eqref{eq:kernel}-\eqref{eq:p_n}, we have
\begin{eqnarray*}
  p_n(\Omega_{n+1}^c) =\frac{\int_{\Omega_n} \left( \int_{\Omega_{n+1}^c} k_n(y,x)  dy \right) d\mu_n(x) }{  \int_{\Omega_n} d\mu_n(x) }   &\leq &  \max_{x\in \myin}  \int_{\Omega_{n+1}^c} k_n(y,x)  dy .  
\end{eqnarray*}

For $n>N_3$ where $N_3$ is the positive integer in the statement and the proof for~\Cref{lem:d_n_vs_Omega_n}, we have $x^+ \in\Omega_{n+1}$ provided that $x\in\Omega_n$. In this case, we also have $B(x^+;d_n(x)) \subseteq \Omega_{n+1}$ where $d_n(x) = \inf_{y\in \Omega_{n+1}^c} |y- x^+|$ (see~\eqref{eq:d_n def}). We then have
\begin{eqnarray*}
  \max_{x\in \myin}  \int_{\Omega_{n+1}^c} k_n(y,x)  dy & \leq  & \max_{x\in \myin}  \frac{1}{  \left(\sqrt{2\pi}\sigma_n\right)^d } \int_{|y-x^+|>d_n(x)}    \exp \left(-\frac{|y-x^+|^2}{2{\sigma}^2_n}\right) dy 
  =  \max_{x\in \myin}  \Prob\left(Q >\frac{d_n(x)^2}{\sigma_n^2} \right), 
\end{eqnarray*}
where the random variable $Q\sim \chi_d^2$, the chi-squared distribution (a sum of the squares of $d$ independent standard normal random variables).  Based on \ref{itm:B2} and~\Cref{lem:d_n_vs_Omega_n}, for any $x\in\myin$, 
\begin{equation}  \label{eq:two-side-bounds}
c^*\sqrt{\log n}  \leq  \frac{d_n(x)}{\sigma_n} = \frac{d_n(x) \sqrt{\log n}}{c_0 \sqrt[d]{|\Omega_n|}}  \leq \sqrt{\log n},\qquad \forall n\geq N_3.
\end{equation}
Therefore, we have
\begin{equation}\label{eq:pn_upper_bd}
 p_n(\Omega_{n+1}^c) \leq \max_{x\in \myin}  \int_{\Omega_{n+1}^c} k_n(y,x)  dy \leq \max_{x\in \myin}  \Prob\left(Q >\frac{d_n(x)^2}{\sigma_n^2} \right) \leq  \Prob\left(Q > (c^*)^2 \log n \right),
\end{equation}
where we dropped ``$\max_{x\in \myin}$'' in the last term of~\eqref{eq:pn_upper_bd} since the final upper bound is $x$-independent.

Based on the properties of the chi-squared distribution, we have the following bounds for different $d$.
\begin{itemize}
    \item When $d=1$, and thus $Q = Z^2$ where $Z\sim \mathcal{N}(0,1)$, we have the Chernoff bound~\cite[Eqn.(5)]{chiani2003new}
    \begin{eqnarray*} 
        \Prob\left(Q >(c^*)^2 \log n \right) &= & 2 \Prob\left(Z >  c^*\sqrt{\log n}\right) \nonumber \\
        & = &
        \text{erfc}\left(\frac{c^* \sqrt{\log n} }{\sqrt{2}}\right) \leq e^{-(c^*)^2 (\log n)/2} = n^{-\frac{(c^*)^2 }{2}},
    \end{eqnarray*}
    where $\text{erfc}(\cdot)$ is the complementary error function.
    \item The cumulative distribution function of the chi-squared distribution is $F(x) = \frac{\gamma(\frac{d}{2}, \frac{x}{2})}{\Gamma(\frac{d}{2})}$ where $\gamma(a,z)$ is the lower incomplete gamma function. When $d=2$, it is well known that $F(x)$ has an analytic form $1-e^{-x/2}$, and thus
    \begin{equation} %
        \Prob\left(Q >(c^*)^2 \log n \right) = e^{-(c^*)^2 (\log n)/2} = n^{-\frac{(c^*)^2 }{2}}.
    \end{equation}
    
    \item When $a>1$ and $z>2(a-1)$ where $z\in\R$, there exists an upper bound for the upper incomplete function $\Gamma(a,z) = \Gamma(a) - \gamma(a,z)$ as $\Gamma(a,z) < 2z^{a-1}e^{-z}$~\cite[Eqn.(3.2)]{natalini2000inequalities}. Based on~\eqref{eq:two-side-bounds}, there exists $N_4\in \N$ such that $\forall n\geq \max(N_3, N_4)$, $(c^*)^2 (\log n) /2 > d-2$.  When $d\geq 3$, by plugging in $a = \frac{d}{2}$ and $z = (c^*)^2 (\log n)/2$, we obtain
\begin{eqnarray}
        \Prob\left(Q >(c^*)^2 (\log n) \right) &<& \frac{2 \left( (c^*)^2 (\log n) /2 \right)^{\frac{d}{2}-1}}{\Gamma(\frac{d}{2})} e^{-\frac{(c^*)^2  \log n}{2}} \nonumber \\
        &=& \frac{2^{2-\frac{d}{2}}}{\Gamma(\frac{d}{2})}\left( (c^*)^2 \log n \right)^{\frac{d}{2}-1} n^{-\frac{(c^*)^2 }{2}} .  \label{eq:3d+_bd}
\end{eqnarray}
Thus, \eqref{eq:3d+_bd} holds if the dimension $d\geq 3$ and $n\geq \max(N_3, N_4)$. Here, $N_4$ is $d$-dependent.
\end{itemize}

These three scenarios cover all possible $d$, the number of dimensions. Combining all the cases above with~\ref{itm:B2}-\ref{itm:B3} and~\eqref{eq:two-side-bounds}, we have
\begin{equation} %
       \frac{p_n(\Omega_{n+1}^c) }{  |\Omega_{n+1}|} \leq 
       \begin{cases}
       \frac{1}{|\Omega_0|}n^{\alpha - (c^*)^2/2}, & d = 1,2,\\
       \ \\
       \frac{ 2^{2-d/2}}{{|\Omega_0|\Gamma(\frac{d}{2})}} \   (\log n)^{\frac{d}{2}-1}\  n^{\alpha - (c^*)^2/2}, &d\geq 3.
       \end{cases}
\end{equation}
Therefore, if we choose $0 <\alpha < \frac{(c^*)^2}{2} $, \eqref{eq:jumpout_vs_jumpin} holds. Also, there exists $N_5\in\N$ such that $\forall n\geq N :=\max\{ N_3, N_4,N_5 \}$, we have $2 (\log n)^{\frac{d}{2}-1} < n^{(c^*)^2/4-\alpha/2}$. Since $\frac{(c^*)^2}{4} + \frac{\alpha}{2} < \frac{(c^*)^2}{2}$, and $2^{1-d/2} <  \Gamma(\frac{d}{2})$ for $d \geq 3$, we can now integrate our results for different values of $d$ and obtain an upper bound
\begin{equation}
    p_n(\Omega_{n+1}^c) \leq  n^{-(c^*)^2/4-\alpha/2}.
\end{equation}
The lower bound in~\eqref{eq:in_out} is based on~\eqref{eq:pn_lower_bd} and~\ref{itm:B3}.
\end{proof}

\subsection{Proof of the Main Theorem}\label{sec:main_proof}
Once we have both~\Cref{thm:prop1} and~\Cref{thm:prop2}, in particular, the upper and lower bounds for $p_n(\Omega_{n+1}^c) = \Prob (X_{n+1}\in\Omega_{n+1}^c|X_n \in \myin)$, we are ready to prove our main result. \Cref{thm:main} demonstrates that the iterates $\{X_n\}$ converge to $x^*$ in probability. That is,  
\begin{equation}\label{eq:convergence}
  \lim_{n\rightarrow\infty}  \Prob \left(|X_n - x^*|  >\epsilon \right) = 0,\qquad \forall\epsilon>0.
\end{equation}

\begin{proof}[Proof of~\Cref{thm:main}]
Since \ref{itm:A1}-\ref{itm:A2} and \ref{itm:B1}-\ref{itm:B3} hold, results from all previous lemmas apply. 

Recall that $A_n$ denotes the event $\{X_n\in\myin\}$. We consider $n\geq N$ where $N$ is the integer in~\Cref{thm:prop2} such that~\eqref{eq:in_out} holds. 
By the law of total probability,
\begin{eqnarray}
 \Prob \left(A_{n+1}^c\right) &= &\Prob(A_{n+1}^c|A_n) \ \Prob(A_n ) +\Prob(A_{n+1}^c|A_n^c) \ \Prob(A_n^c ) \nonumber \\
 &\leq & \beta_2 n^{-\alpha_2} (1 - \Prob(A_n^c ))   + (1- \kappa n^{-\alpha})    \Prob(A_n^c ) \nonumber \\
 & = & g(n)  \Prob(A_n^c ) + h(n),\label{eq:recur1}
\end{eqnarray}
where $0 < \alpha < \alpha_2 < \frac{(c^*)^2}{2} < 1$,
$g(n) = 1- \kappa n^{-\alpha} -  \beta_2 n^{-\alpha_2}$ and $h(n) =\beta_2 n^{-\alpha_2}$.

Consider the following first-order non-homogeneous recurrence with variable coefficients
\[
r_{n+1} = g(n)r_{n} + h(n) 
\]
starting with $n = n_0$ and $r_{n_0} =   \Prob \left(A_{n_0}^c\right)$, where $n_0 = \min \{n\in\N| n\geq N, g(n) > 0\}$.
Based on~\eqref{eq:recur1}, $\Prob \left(A_n^c\right) \leq r_n, \forall n\geq n_0$, because $g(n)$ is monotonically increasing and $g(n_0) > 0$.
Thus, it is sufficient to show that $\lim\limits_{n\rightarrow\infty} r_n = 0$, which directly implies $\lim\limits_{n\rightarrow\infty}  \Prob \left(A_n^c\right)  = 0$.

The recurrence $\{r_n\}$ has an analytic solution~\cite[Section 3.3]{rosen2017handbook}. That is,
\begin{eqnarray}
    r_n &=& \left(\prod_{k=n_0}^{n-1} g(k) \right) \left(r_{n_0} + \sum_{m=n_0}^{n-1}\frac{h(m)}{\prod_{k=n_0}^m g(k)}\right) \nonumber \\
    &= & \left(\prod_{k=n_0}^{n-1}  \left(1- \kappa k^{-\alpha} -  \beta_2 k^{-\alpha_2} \right) \right) \left(r_{n_0} +  \sum_{m=n_0}^{n-1}\frac{ \beta_2 m^{-\alpha_2}}{\prod_{k=n_0}^m  \left(1- \kappa k^{-\alpha} -  \beta_2 k^{-\alpha_2} \right)}\right)  \nonumber \\
    & =  &  \underbracket{r_{n_0} \left(\prod_{k=n_0}^{n-1}  \left(1- \kappa k^{-\alpha} -  \beta_2 k^{-\alpha_2} \right) \right)}_{I_1(n)} + \underbracket{\beta_2 \sum_{m=n_0}^{n-1}  m^{-\alpha_2}\prod_{k=m+1}^{n-1} \left(1- \kappa k^{-\alpha} -  \beta_2 k^{-\alpha_2} \right)}_{I_2(n)}.\nonumber
\end{eqnarray}
Note that for any $k_0\in\N$,
\begin{multline}
 \sum_{k=k_0}^{n-1}  \log \left(1- \kappa k^{-\alpha} -  \beta_2 k^{-\alpha_2} \right)  \leq - \sum_{k=k_0}^{n-1} \left(\kappa k^{-\alpha} + \beta_2 k^{-\alpha_2} \right) \leq  - \int_{k_0}^{n} \left( \kappa x^{-\alpha } +\beta_2 x^{-\alpha_2} \right)dx  \label{eq:logexp} \\
  =  \frac{\kappa \left(k_0^{1-\alpha} - n^{1-\alpha}\right)}{1-\alpha} + \frac{\beta_2 \left(k_0^{1-\alpha_2} - n^{1-\alpha_2}\right)}{1-\alpha_2}. 
\end{multline}
With  $z_1 := \frac{\kappa}{1-\alpha}$ and  $z_2 := \frac{\beta_2}{1-\alpha_2}$, we have the following by plugging $k_0 = n_0$ into~\eqref{eq:logexp}.
\begin{equation}
\lim_{n\rightarrow\infty}I_1(n) \leq r_{n_0} \exp(z_1 n_0^{1-\alpha})\exp(z_2 n_0^{1-\alpha_2})\lim_{n\rightarrow\infty} \exp \left( -z_1 n^{1-\alpha}- z_2  n^{1-\alpha_2}\right)= 0. \nonumber
\end{equation}
We remark that $I_1(n)$ decays to zero exponentially.

Next, we will provide an asymptotic decay rate for $I_2(n)$. Plugging $k_0 = m+1$ into~\eqref{eq:logexp}, 
\begin{eqnarray}
I_2(n) &\leq& \beta_2 e^{ -z_1 n^{1-\alpha}}e^{- z_2  n^{1-\alpha_2}} \sum_{m=n_0}^{n-1}  m^{-\alpha_2} 
e^{z_1 (m+1)^{1-\alpha}} e^{z_2  (m+1)^{1-\alpha_2}}\nonumber\\
&\leq & \underbracket{\beta_2 e^{ -z_1 n^{1-\alpha}}  e^{- z_2  n^{1-\alpha_2}} \int_{n_0}^{n}  s^{-\alpha_2}  e^{z_1 (s+1)^{1-\alpha}} e^{z_2  (s+1)^{1-\alpha_2}}  ds}_{K(n)}. \nonumber
\end{eqnarray} 
We denote the right-hand side of the last inequality as $K(n)$. Consider $Z(n) = \beta_2 \kappa^{-1}n^{\alpha-\alpha_2}$. We may generalize $Z(n)$ and $K(n)$ to functions $Z(x)$ and $K(x)$. Using L'Hospital's rule,
\begin{eqnarray*}
&& \lim_{n\rightarrow\infty} \frac{K(n)}{Z(n)}=  \lim_{x\rightarrow\infty} \frac{K(x)}{Z(x)} \\
 &=&  \lim_{x\rightarrow\infty} \frac{ \beta_2\int_{n_0}^{x}  s^{-\alpha_2}  e^{z_1 (s+1)^{1-\alpha}} e^{z_2  (s+1)^{1-\alpha_2}}  ds }{ \left( \beta_2 \kappa^{-1}x^{\alpha-\alpha_2} \right) e^{ z_1 x^{1-\alpha}}  e^{z_2  x^{1-\alpha_2}} } \\
 &=&  \lim_{x\rightarrow\infty} \frac{ \kappa x^{-\alpha_2}   e^{z_1 (x+1)^{1-\alpha}} e^{z_2  (x+1)^{1-\alpha_2}}  }{x^{\alpha-\alpha_2}e^{ z_1 x^{1-\alpha}}  e^{z_2  x^{1-\alpha_2}}  \left(  (\alpha-\alpha_2) x^{-1} + z_1(1-\alpha)x^{-\alpha} + z_2(1-\alpha_2)x^{-\alpha_2}  \right)}\\
 &=& \lim_{x\rightarrow\infty}  \frac{\kappa x^{-\alpha} }{    (\alpha-\alpha_2) x^{-1} + \kappa x^{-\alpha} + \beta_2x^{-\alpha_2} } \lim_{x\rightarrow\infty} \frac{e^{ z_1  (x+1)^{1-\alpha} }   }{e^{ z_1 x^{1-\alpha} }       } \lim_{x\rightarrow\infty}\frac{e^{ z_2  (x+1)^{1-\alpha_2} }   }{e^{ z_2 x^{1-\alpha_2} }       }
 = 1.
\end{eqnarray*}
We have just proved that $K(n) \sim \beta_2 \kappa^{-1}n^{\alpha-\alpha_2}$. Thus, $r_n = \mathcal{O}(n^{\alpha-\alpha_2})$ and $\lim\limits_{n\rightarrow \infty} r_n = 0$. 

Based on~\eqref{eq:Omega_n_area_3}, the sublevel set $\Omega_n$ is  contained in a slightly bigger ellipsoid $\overline{E}_n$, which is further contained in a ball centered at $x^*$ with a radius $\sqrt{2( {f_n}-\fkf^*)/b_1}$~\cite[P.~462]{boyd2004convex}. Meanwhile,~\eqref{eq:cutoff_bound} provides an upper bound $\sqrt{{{f_n}-\fkf^*}} \leq \mathfrak{c} n^{-\alpha/d}$ where $\mathfrak{c}^d = b_\gamma^{-1/2}  c_3^{-1}|\mathfrak{X}|$. Thus, with $\mathfrak{c}_1:= \sqrt{2/b_1}\mathfrak{c}$, we have
\begin{equation} \label{eq:diam}
  |x-x^*| \leq \sqrt{2 ( {f_n}-\fkf^*)/b_1}\leq \mathfrak{c}_1 n^{-\alpha/d},\quad \forall x\in\Omega_n.
\end{equation} 
For $n\geq n_1$ where $n_1 = \min \{n\in \N| n\geq n_0, I_1(n) < I_2(n), K(n) < 2 Z(n)\}$, we have
\begin{equation} \label{eq:bound_prob}
    \Prob(X_n \not \in\Omega_n) =  \Prob(A_n^c)  \leq r_n < 2I_2(n)   \leq 4\beta_2 \kappa^{-1}n^{\alpha-\alpha_2} = \mathfrak{c}_2 n^{\alpha-\alpha_2}, \quad \mathfrak{c}_2:=4\beta_2 \kappa^{-1}.
\end{equation}

Next, we combine~\eqref{eq:diam} and~\eqref{eq:bound_prob}. Given $\epsilon>0$, $\forall n\geq \mathcal{N}_1 = \max \left(n_1,\lceil(\epsilon\mathfrak{c}_1^{-1} )^{-d/\alpha} \rceil\right)$,
\[
  \Prob \left( |X_n-x^*|  > \epsilon \right)   \leq \Prob \left( |X_n-x^*|  > \mathfrak{c}_1 n^{-\alpha/d} \right) \leq     \Prob \left( X_n \not\in\Omega_n \right)  \leq \mathfrak{c}_2 n^{\alpha-\alpha_2},
\]
which indicates convergence of the iterates $\{X_n\}$ in probability~\eqref{eq:convergence}. 

Similarly, given any $\delta>0$, for all $n\geq \mathcal{N}_2 = \max \left(n_1,\lceil(\delta\mathfrak{c}_2^{-1} )^{\alpha_2 - \alpha} \rceil\right)$, we have
\[
 \Prob \left( |X_n-x^*|  > \mathfrak{c}_1 n^{-\alpha/d} \right) \leq \mathfrak{c}_2 n^{\alpha-\alpha_2} \leq \delta,
\]
which illustrates the rate of algebraic convergence.
\end{proof}

Let us conclude this section with the following remark. Our focus here is on the challenge of proving convergence of the sequence $\{X_n\}_{n\ge 0}$ generated by the algorithm shown in Equation~\eqref{eq:SGD101}. In the literature, very often, the easier problem of convergence for the sequence 
\begin{equation}\label{EQ:min seq}
    \widetilde X_n := \argmin_{x \in \{X_0, X_1,\ldots,X_n\} } f(x)
\end{equation}
is studied. For such a sequence $\{\widetilde X_n \}$, we can derive the following stronger result for the algorithm under the conditions of~\Cref{thm:main}.
\begin{corollary}
Under the same assumptions as in \Cref{thm:main}, the sequence $\{\widetilde X_n\}_{n\ge 0}$ defined in~\eqref{EQ:min seq} converges almost surely to $x^*$, i.e.,
\[
\bbP\left( \lim_{n\rightarrow \infty} \widetilde X_n = x^*\right) = 1.
\]
\end{corollary} 
\begin{proof}
This is a direct consequence of the fact that $\{f(\widetilde X_n)\}_{n\ge 0}$ is monotone decreasing and bounded from below by $f(x^*) = \fkf^* $. It thus has a limit, say, $\hat{f}$. By ``Property One'', i.e., the statement in~\Cref{thm:prop1}, we know $\lim_{n\rightarrow \infty} f(\widetilde X_n) = \hat{f} = \fkf^*$ with probability $1$. By Assumption~\ref{itm:A1}-\ref{itm:A2}, $f$ is locally a bijection. Therefore, $\lim_{n\rightarrow \infty}\widetilde X_n= x^*$ with probability $1$.
\end{proof}

\section{Parameter Estimations and Algorithmic Details} \label{sec:extension}
The main contribution of this paper is to show that it is possible to achieve an algebraic rate of convergence in probability and location by using an SGD algorithm with a state-dependent variance when finding the global minimizer of a nonconvex optimization problem. We have presented the proposed algorithm~\eqref{eq:diffdisc1} and rigorously proved the convergence of the algorithm in~\Cref{sec:proof}. The proof is valid under quite general conditions on the objective function, which is further supported by the numerical results later in~\Cref{sec:example1}. 

For an explicit implementation of the proposed SGD algorithm, we also need to obtain the two essential components that directly define the proposed SGD algorithm  in~\eqref{eq:diffdisc_new}: the state-dependent variance function $\sigma_n(f(x))$ and the step size $\eta_n$. In particular, the state-dependent variance in~\eqref{eq:diffdisc1}  is directly determined by two scalar variables that change with the iteration number $n$: $f_n$ (the cutoff value that determines the sublevel-set $\Omega_n$), ${\sigma}_n$ (the decaying variance inside $\Omega_n$).

There are many types of heuristic strategies to set $\eta_n$, $f_n$, and $\sigma_n$, such as parameter tuning, which is common in implementing practical algorithms for complex optimization problems. Some of the strategies are presented in the numerical examples (see~\Cref{sec:example1}). In this section, we discuss a few particular approaches to set these scalars.

\subsection{Given Estimates of \texorpdfstring{$\Omega^*(\fkf)$}{}, \texorpdfstring{$b_1$}{} and \texorpdfstring{$b_2$}{}}\label{sec:extension1}
First, assume that we have already obtained good estimates for the function $\Omega^*: \R \mapsto \R^+$, 
\begin{equation}\label{eq:Omega*}
\Omega^*(\fkf) = |\{x\in \mathfrak{X}|f(x)\leq \fkf \}|,
\end{equation}
which maps the cutoff objective function value $\fkf$ to the volume of the $\fkf$-sublevel set of the objective function $f(x)$. Note that $\Omega^*(\fkf)$ is a monotonically increasing function of $\fkf$. Mathematically, we are given two monotonically increasing functions $\underline{\Omega}(\fkf)$ and $ \overline{\Omega}(\fkf)$ such that $\forall \fkf$,
\begin{equation*}
 \underline{\Omega}(\fkf) \leq \underline{\gamma} \Omega^*(\fkf) \leq \Omega^*(\fkf)  \leq \overline{\gamma}  \Omega^*(\fkf)  \leq  \overline{\Omega}(\fkf),
\end{equation*}
where $0<\underline{\gamma} \leq 1 \leq \overline{\gamma} <\infty$. Moreover, we assume to have the upper and lower bounds for $b_1$ and $b_2$; see \ref{itm:A1} for their definitions. That is, we are given two estimated scalars $\underline{b_1}$ and $\overline{b_2}$ such that
\begin{equation*}
     0 < \underline{b_1} \leq b_1 \leq b_2 \leq \overline{b_2} < \infty.
\end{equation*}

Recall that Assumption~\ref{itm:B3} requires a pre-determined decay rate for the volume of the $f_n$-sublevel set, i.e., $|\Omega_n|$, in terms of the iteration number $n$. Our main result \Cref{thm:main} and its proof in~\Cref{sec:proof} are based upon this assumption. Utilizing the estimations of function $\Omega^*(\fkf)$, we could set the sequence $\{f_n\}$ for each $n$ by choosing the inverse of function $\overline{\Omega}(\fkf)$ evaluated at the sequence $\{|\Omega_n|\}$ based on the pre-determined decay rate in~\ref{itm:B3} using estimations $\underline{b_1}$ and $\overline{b_2}$. Since $\overline{\Omega}(\fkf)$ is monotonically increasing, its general inverse function is well-defined, and the estimated $\{f_n\}$ will be monotonically decreasing with respect to $n$. The value can be practically computed through numerical interpolation. Once we have the pre-determined values for $\{|\Omega_n|\}$, the variable ${\sigma}_n$ can be set accordingly in order to satisfy Assumption~\ref{itm:B2}. Based on Assumption~\ref{itm:B1}, we may choose the step size as a constant $\eta_n = \eta < 2/\overline{b_2}$. As a result, all three essential assumptions \ref{itm:B1}-\ref{itm:B3} are satisfied, for which~\Cref{thm:main} will follow.

\subsection{Estimations of \texorpdfstring{$\Omega^*(\fkf)$}{}, \texorpdfstring{$\underline{b_1}$}{} and \texorpdfstring{$\overline{b_2}$}{}}\label{sec:extension3}
In~\Cref{sec:extension1}, we rely on good estimations of $\Omega^*(\fkf)$, $\underline{b_1}$  and $\overline{b_2}$. Here, we discuss several ways to obtain estimations of the required quantities that were assumed to be known earlier. %

\subsubsection{Monte Carlo Estimation of \texorpdfstring{$\Omega^*(\fkf)$}{}} \label{sec:omega_n_est}
Recall that the sublevel set $\Omega_n$ is associated with the small variance $\sigma_n$. Since the noise power of the proposed SGD algorithm~\eqref{eq:diffdisc1} inside the low-variance subset, characterized by ${\sigma}_n$, depends on the Lebesgue measure of the sublevel set $\Omega_n$, which is implicitly controlled by the cutoff value $f_n$, it is thus important to estimate the relation between $f_n$ and $|\Omega_n|$. That is, the function $\Omega^*(\fkf)$ defined in~\eqref{eq:Omega*}, if not known a priori.

Let $\fkf_{max} = \max\limits_{x\in \mathfrak{X}} f(x)$. Given any scalar $\fkf\in [\fkf^*,\fkf_{max}]$, we define
\begin{equation} \label{eq:G_cdf}
     \mathcal{G}(\fkf) := \frac{\Omega^*(\fkf)}{|\mathfrak{X}|} = \int_\mathfrak{X} \mathbbm{1}_{f(x) \leq \fkf} \ p(x) dx =\int_\mathfrak{X} \mathfrak{I}_\fkf(x)p(x) dx,\quad \mathfrak{I}_\fkf(x) :=\mathbbm{1}_{f(x) \leq \fkf},
\end{equation}
where the probability density function $p(x) = 1/|\mathfrak{X}|$, denoting the uniform distribution over the domain $\mathfrak{X}$. We may estimate $\mathcal{G}(\fkf)$ by a Monte Carlo quadrature denoted as $\mathcal{F}_N(\fkf)$ where
\begin{equation} \label{eq:F_cdf}
\mathcal{F}_N(\fkf) = \frac{1}{N} \sum_{i=1}^N  \mathbbm{1}_{f(Y_i) \leq \fkf} = \frac{1}{N}\sum_{i=1}^N \mathfrak{I}_\fkf(Y_i).
\end{equation}
Here, $\{Y_i\}$ are $N$ i.i.d.\ samples following $p(x)$, i.e., uniform samples from the domain $\mathfrak{X}$. So far, we have two cumulative distribution functions $\mathcal{F}_N$ and $\mathcal{G}$ where $\mathcal{F}_N$ is regarded as a Monte Carlo estimation of $\mathcal{G}$. Based on classic theories of the Monte Carlo methods~\cite{caflisch1998monte}, %
\begin{equation}\label{eq:rv1}
    \mathcal{F}_N( \fkf)\approx \mathcal{G}(\fkf) + \frac{\sigma(\mathfrak{I}_\fkf)}{\sqrt{N}} \nu_1, \quad \forall \fkf\in [\fkf^*, \fkf_{max}],
\end{equation}
where $\nu_1\sim \mathcal{N}(0,1)$ and the the standard deviation $\sigma(\mathfrak{I}_\fkf)$ satisfies
\begin{equation}\label{eq:std_f}
    \sigma(\mathfrak{I}_\fkf)^2 =  \int_\fX \mathfrak{I}_\fkf^2\ p(x) dx -  \mathcal{G}(\fkf)^2  =\mathcal{G}(\fkf) - \mathcal{G}(\fkf)^2  \leq \frac{1}{4}.
\end{equation}

For our proposed SGD algorithm~\eqref{eq:diffdisc1}, we need a sequence of cutoff function values $\{f_n\}$ such that $\mathcal{G}(f_n) = n^{-\alpha}$ for a given $\alpha>0$, as required by Assumption~\ref{itm:B3}. Based on the estimated function $\mathcal{F}_N$, we obtain a sequence of estimation $\{\widehat{f}_{n,N}\}$ such that
\begin{equation}~\label{eq:f_n_est}
    \mathcal{F}_N(\widehat{f}_{n,N} ) = \mathcal{G}(f_n) = n^{-\alpha}.
\end{equation}
Next, we will prove that the estimated cutoff value $\widehat{f}_{n,N}$ will converge in probability to the true $f_n$ as the number of samples $N\rightarrow\infty$.

We start with the following assumptions. 
\begin{enumerate}[label = \textbf{D\arabic*}]
\item \label{itm:D1} There exists $\tilde \fkf$ such that $\{x: f(x) \leq \tilde \fkf\} \subseteq \Omega_{sc}$ and the objective function $f(x)\in C^3(\Omega_{sc})$ is locally quadratic.
\item \label{itm:D2} For any $\fkf \in (\tilde \fkf , \fkf_{max})$, there exists a  constant $\underline{L}$ such that $\inf\limits_{\delta>0} \frac{\mathcal{G}(\fkf+\delta) - \mathcal{G}(\fkf)}{\delta} > \underline{L} >0$.
\end{enumerate}
\begin{theorem}\label{thm:MC_justificaiton}
Let \ref{itm:A1}, \ref{itm:A2}, \ref{itm:B3}, \ref{itm:D1} and~\ref{itm:D2} hold. Consider $\widehat{f}_{n,N}$ defined in~\eqref{eq:f_n_est} with $N\geq n$, and a given $\alpha \in (0,1)$. Then for all $\epsilon >0$, 
\begin{equation}\label{eq:omega_n_prob}
\lim_{N\rightarrow\infty} \Prob\left(|\widehat{f}_{n,N} - f_n| < \epsilon \right)  = 1.
\end{equation} 
\end{theorem}
\begin{proof}
Based on Assumption~\ref{itm:D2}, for any two $\fkf_1,\fkf_2 \in (\tilde \fkf, \fkf_{max})$, we have
\begin{eqnarray}
| \mathcal{F}_N( \fkf_1) -   \mathcal{F}_N( \fkf_2 )|  &=& 
  |\mathcal{F}_N( \fkf_1) -   \mathcal{G}( \fkf_1) +   \mathcal{G}( \fkf_1) -   \mathcal{G}( \fkf_2) +   \mathcal{G}( \fkf_2) -  \mathcal{F}_N( \fkf_2 ) | \nonumber \\
  & >& \underline{L} |\fkf_1-\fkf_2|  - |\mathcal{F}_N( \fkf_1) -   \mathcal{G}( \fkf_1)| - |\mathcal{F}_N( \fkf_2) -   \mathcal{G}( \fkf_2)|. \label{eq:bounding F_n}
\end{eqnarray}
Furthermore, there exists $\mathfrak{n}\in \mathbb N$ such that $f_{\mathfrak{n}+1} \leq \tilde \fkf < f_\mathfrak{n}$ where $\{f_n\}$ are defined by~\eqref{eq:f_n_est}. Clearly, $\tilde \fkf < f_n$ for any $n \leq \mathfrak n$. With $\fkf_1$ replaced by $\widehat{f}_{n,N}$, and $\fkf_2$ replaced by $f_n$, \eqref{eq:bounding F_n} yields
\begin{eqnarray}
 |\widehat{f}_{n,N}-f_n| &\leq& \frac{1}{\underline{L}} \left( | \mathcal{F}_N( \widehat{f}_{n,N}) -   \mathcal{F}_N( f_n )| +  | \mathcal G( f_n ) -   \mathcal{F}_N( f_n )|  +  | \mathcal G( \widehat{f}_{n,N}) -   \mathcal{F}_N( \widehat{f}_{n,N} )|  \right) \nonumber \\
 &=& \frac{1}{\underline{L}} \left( 2| \mathcal{G}(f_n) -   \mathcal{F}_N( f_n )| +  | \mathcal G( \widehat{f}_{n,N}) -   \mathcal{F}_N( \widehat{f}_{n,N} ) \right). \nonumber
\end{eqnarray}
As a result, for any $\epsilon>0$, we have the following by Chebyshev’s inequality and~\eqref{eq:std_f}.
\begin{eqnarray}
&& \Prob( |\widehat{f}_{n,N}-f_n| <  \epsilon ) \label{eq:fn_large} \\
&\geq & 
\Prob( \{ | \mathcal{G}(f_n) -   \mathcal{F}_N( f_n )| <   \epsilon/(3\underline{L})| \} 
\cap  
\{\mathcal{G}(\widehat{f}_{n,N}) -   \mathcal{F}_N( \widehat{f}_{n,N} )| < \epsilon/(3\underline{L})\} ) 
\nonumber \\
&\geq& 1 - \Prob ( | \mathcal{G}(f_n) -   \mathcal{F}_N( f_n )| \geq   \epsilon/(3\underline{L})|  -  \Prob (  \mathcal{G}(\widehat{f}_{n,N}) -   \mathcal{F}_N( \widehat{f}_{n,N} )| \geq \epsilon/(3\underline{L}) )
 \geq 1-  \frac{9\underline{L}^2}{2N\epsilon^2}. \nonumber 
\end{eqnarray}

When $n> \mathfrak{n}$, we have $f_n \leq \tilde \fkf$ and thus $\Omega_n \subseteq \{x:f(x)\leq \tilde \fkf\}\subseteq \Omega_{sc}$ by~\ref{itm:D1}. We also have that $f(x)\in C^3(\Omega_{sc})$ according to~\ref{itm:D1}, and
\begin{equation}\label{eeq:f_in_Omega_sc}
f(x) = f^* + (x-x^*)^\top H (x-x^*) + o(|x-x^*|^2), \quad \forall x\in \Omega_{sc},
\end{equation}
where $H = \nabla^2 f(x^*)$ is positive definite. Thus, $\mathcal{G} (\fkf )$ is continuously differentiable for $\fkf \leq \tilde \fkf$. According to~\eqref{eq:omega_n_volume} and~\eqref{eq:f_n_est}, $n^{-\alpha} = \mathcal{G} (f_n) = \mathcal{O}\left((f_n - \fkf^*)^{d/2} \right)$, which leads to $ f_n - \fkf^* =\mathcal{O} \left( n^{-2\alpha/d}\right)$. Consequently, 
$\mathcal{G}'(f_n)  = \mathcal{O}\left((f_n - \fkf^*)^{d/2-1} \right) =\mathcal{O}\left(n^{-\alpha(1-2/d)} \right)$.
We remark here that as $n\rightarrow \infty$, $\mathcal{G}'(f_n)\rightarrow\infty$ for $d=1$, remains constant for $d=2$, and goes to zero for $d\geq 3$. Here, analyzing the error $|\widehat f_{n,N} - f_n|$ links to the sensitivity analysis for the root-finding problem~\cite[Chapter 3]{corless2013graduate}. We apply the classical result and obtain
\begin{equation}
    |\widehat f_{n,N} - f_n| \approx \bigg| \frac{  \mathcal{G} (f_n) -   \mathcal{F}_N (f_n)}{ \mathcal{G}'(f_n)} \bigg|. %
\end{equation}
Finally, for any $\epsilon>0$, by Chebyshev’s inequality and \eqref{eq:std_f}, we have
\begin{equation}\label{eq:fn_small}
 \Prob( |\widehat{f}_{n,N}-f_n| \geq  \epsilon ) \approx \Prob \left( | \mathcal{G}(f_n) -   \mathcal{F}_N( f_n )| \geq  \mathcal{G}'(f_n)\epsilon \right)  
 \leq  \frac{ n^{2\alpha(1-\frac{2}{d})}}{C N \epsilon^2},
\end{equation}
where $C>0$ is a constant. When $d\leq 2$, the numerator of the last term monotonically decreases, so $\Prob( |\widehat{f}_{n,N}-f_n| \geq  \epsilon )$ decreases at least with the rate $\mathcal{O}(N^{-1})$. When $d>2$, we have $\alpha < \frac{1}{2}$ by Assumption~\ref{itm:B3}, and thus $2\alpha(1-\frac{2}{d}) < 1$. We also have $\Prob( |\widehat{f}_{n,N}-f_n| \geq  \epsilon )$ decreases as $N\rightarrow \infty$ since $n\leq N$ by assumption.

The proof is complete by combining~\eqref{eq:fn_large} and \eqref{eq:fn_small}.
\end{proof}

\begin{remark}
\Cref{eq:std_f} shows that the absolute error in the Monte Carlo estimation is small when $\Omega^*(\fkf)\rightarrow 0$ (or equivalently when $\fkf\rightarrow \fkf^* $) as well as when $\Omega^*(\fkf)\rightarrow |\fX|$ (or equivalently when $\fkf\rightarrow \fkf_{max}$). On the other hand, the error is the largest when $\Omega^*(\fkf) = |\fX|/2$. This observation matches the numerical results shown in~\Cref{fig:cdf-compare}.
\end{remark}

Although the proof in~\Cref{sec:main} is based on the setup~\eqref{eq:diffdisc1}, it is sometimes more convenient to draw a sample with high-variance variance $\sigma_n^+$ rather than a uniform sample from $\fX$ when $X_n\in\Omega_n^c$. In that case, the Monte Carlo estimation described above requires many additional objective function evaluations, which could be computationally prohibitive for large-scale problems. An alternative approach is to utilize the iterates $\{X_n\}$ that are already generated as we execute the proposed SGD algorithm~\eqref{eq:diffdisc_new}. This is what we refer to as the online estimation. However, based on~\eqref{eq:diffdisc_new}, the iterates $\{X_n\}$ are neither independent nor following the uniform distribution. Nevertheless, one could still \textit{estimate} the relation between $\fkf$ and $\Omega^*(\fkf)$ using the high-variance iterates. Here, we say $X_{n+1}$ is a high-variance iterate if $X_{n}\in\Omega_n^c$ and thus $\Prob(X_{n+1}|X_{n})$ is subject to the large noise coefficient $\sigma_n^+$ in~\eqref{eq:two-stage-sigma}. 
The larger the noise coefficient $\sigma_n^+$ is, the less the iterate $X_{n+1}$ is correlated with $X_{n}$, which makes $X_{n+1}$ closer to being a uniform sample drawn from the domain $\fX$.
Later in~\Cref{sec:estimate_fn}, we will present a numerical example to demonstrate the effectiveness of such online estimation of $\Omega^*(\fkf)$ by using the high-variance iterates compared with the estimation using the i.i.d.\ uniform samples.

\subsubsection{Estimating \texorpdfstring{$b_1$}{} and  \texorpdfstring{$b_2$}{}}
\label{sec:b1b2_est}

Note that $b_1$ and $b_2$ are lower and upper bounds for the eigenvalues of $\nabla^2 f(x), x\in\Omega_{sc}$. Since $f(x)$ is locally quadratic in $\Omega_{sc}$ based on~\eqref{eeq:f_in_Omega_sc}, we may regard  the smallest and the largest eigenvalues of $H = \nabla^2 f(x^*)$ as $b_1$ and $b_2$, respectively. 

Given any set of scalar values $\{g_n\}$ where $g_n\xrightarrow{n\rightarrow\infty} \fkf^*$, for $n$ large enough, we have
\[
\Pi_n : = \{x:f(x) \leq g_n\} \subseteq B_{\Pi_n}  := \{x: |x-x^*| \leq \sqrt{2(g_n - \fkf^*) / \lambda_1} \},
\]
where $0<\lambda_1\leq \ldots \leq \lambda_d <\infty$ are eigenvalues of  $H$. The diameter of $B_{\Pi_n}$ is  $2\sqrt{2(g_n - \fkf^*) / \lambda_1}$, which could be approximated by $\operatorname{diam}(\Pi_n)$. Recall that $\operatorname{diam}(A) = \sup\limits_{y_1,y_2\in A} |y_1-y_2|$. Thus,
\begin{equation} \label{eq:b1_def_motive}
    b_1 \approx \lambda_1 = \lim_{n\rightarrow \infty} \frac{8(g_n-f^*)}{|\operatorname{diam}(B_{\Pi_n})|^2} \approx \lim_{n\rightarrow \infty} \frac{8(g_n-f^*)}{\sup\limits_{f(y_1),f(y_2) \leq g_n}|y_1-y_2|^2 }.
\end{equation}
On the other hand, we may approximate the largest eigenvalue $\lambda_d$ by the maximum slope of the gradient function $G(x)$ at $x=x^*$. That is,
\begin{equation} \label{eq:b2_def_motive}
 b_2 \approx \lambda_d =  \sup\limits_{y\neq 0} \frac{y^\top H y }{y^\top y} =\lim_{n\rightarrow \infty} \sup\limits_{f(y_1),f(y_2) \leq g_n} \frac{|G(y_1)-G(y_2)|}{|y_1-y_2|}.
\end{equation}

Note that estimations based on~\eqref{eq:b1_def_motive} and~\eqref{eq:b2_def_motive} are only valid once we have access to samples from $\Omega_{sc}$. On the other hand, the accuracy of $b_1$ and $b_2$ estimations only become relevant to the convergence rate of the SGD algorithm once $\Omega_n\subseteq \Omega_{sc}$. Taking both aspects into account, we propose the following iterative algorithm to approximate $b_1$ and $b_2$.

Consider i.i.d.~samples $\{Y_n\}$ drawn uniformly from the domain $\mathfrak{X}$, and the set of scalar values $\{g_n\}$ where $g_n\rightarrow \fkf^*$. For a given $m\in\mathbb{N}$, we first select $m$ samples where
$$
\{y_{n_1},y_{n_2}, \ldots, y_{n_m}\} \subseteq \{y_{1},y_{2}\ldots, y_{n_m}\}
$$
such that $f(Y_{n_i})\leq g_1$ for $i=1,\ldots,m$. Using those $m$ samples, we estimate $b_1$ and $b_2$ based on~\eqref{eq:b1_def_motive} and~\eqref{eq:b2_def_motive}. For $k=1,\ldots,m-1$,
\begin{eqnarray*}
 \widehat{b_2}^{(1)}   &\leftarrow&  \max \Big( \widehat{b_2}^{(1)}, \frac{|G(Y_{n_{k+1}}) - G(Y_{n_k})|}{|Y_{n_{k+1}} - Y_{n_k}|}  \Big), \\
\widehat{\fkf^*} &\leftarrow&  \min \big( \widehat{\fkf^*},  f(Y_{n_k})\big), \\
 \widehat{D}^{(1)}   &\leftarrow&  \max \big( \widehat{D}^{(1)}, |Y_{n_{k+1}} - Y_{n_k}| \big),
\end{eqnarray*}
and finally, with the best $\widehat{\fkf^*}$ and $\widehat{D}^{(1)}$ estimated from the $m$ samples, we set
\[
  \widehat{b_1}^{(1)}  = 8|g_1 - \widehat{\fkf^*}|/ \widehat{D}^{(1)}.
\]
We may then use $\widehat{b_1}^{(1)}$ and $\widehat{b_2}^{(1)}$ to obtain $\alpha^{(1)}$ as an estimator for the true $\alpha$ based on~\ref{itm:B3}, followed  by estimations of $f_n$ as discussed in~\Cref{sec:omega_n_est}.

Since we do not know where the set $\Omega_{sc}$ is a priori, the estimation process has to be repeated to obtain accurate approximations of $b_1$ and $b_2$. We summarize the algorithm in~\Cref{alg:b1b2_2}.
\begin{algorithm}[ht!]
\caption{Estimation of $b_1$ and $b_2$~\label{alg:b1b2_2}}
\begin{algorithmic}[1]
\STATE Given i.i.d.~uniform samples $\{Y_n\}$, $m\in\mathbb{N}$, and scalar values $\{g_n\}$ where $g_n\rightarrow \fkf^*$. Set $\widehat{\fkf^*}  = f(Y_1)$.
\FOR{$l= 1,2,\ldots $}
\STATE Set $k_l = ml(l-1)/2$ and $\widehat{b_2}^{(l)} = \widehat{D}^{(l)} = i=0$.
\WHILE{$i < ml$}
\IF{$f(Y_n) < g_{1+k_l}$}
\STATE Update $i \leftarrow i+1$, and set $ Y_{n_{s}} = Y_n$ where $s= i+k_l$.
\STATE Update $\widehat{\fkf^*} \leftarrow  \min \big ( \widehat{\fkf^*},  f(Y_{n_{s}})\big )$.
\IF{$i\geq 2$}
\STATE Update $ \widehat{D}^{(l)}\leftarrow \max \big( \widehat{D}^{(l)}, |Y_{n_{s}} - Y_{n_{s-1}}| \big)$,  and $\widehat{b_2}^{(l)}   \leftarrow  \max \Big( \widehat{b_2}^{(1)}, \frac{|G(Y_{n_{s}}) - G(Y_{n_{s-1}})|}{|Y_{n_{s}} - Y_{n_{s-1}}|}  \Big)$.
\ENDIF
\ENDIF
\ENDWHILE
\STATE Set $\widehat{b_1}^{(l)}  = 8|g_{1+k_l} - \widehat{\fkf^*}|/ \widehat{D}^{(l)}$. Obtain $\alpha^{(l)}$ based on $\widehat{b_1}^{(l)}$ and $\widehat{b_2}^{(l)}$.
\STATE Estimate $\{f_n\}$ for $n = 1 + k_{l}, 2 + k_{l},\ldots, k_{l+1}$.
\ENDFOR
\end{algorithmic}
\end{algorithm}

\subsubsection{Estimating \texorpdfstring{$\underline{b_1}$}{} and  \texorpdfstring{$\overline{b_2}$}{}}\label{sec:bound_est}
While~\Cref{sec:b1b2_est} discusses how to approximate $b_1$ and $b_2$ accurately by repeating the process with selected uniform samples that are gradually more likely to be drawn from $\Omega_{sc}$, we present here an algorithm that provides crude lower and upper bounds for $b_1$ and $b_2$, respectively. We remark that although this algorithm is more straightforward, it will result in a slower algebraic convergence rate due to the possible severe under- and over-estimations.

Based on~\eqref{eq:b1_def_motive}, an upper bound for $b_2$, denoted as $\overline{b_2}$, can be obtained based on the divided difference. Given two consecutive iterates $X_n$ and $X_{n+1}$, we compute
\begin{equation}\label{eq:Lip_constant}
    E_n = \frac{|G(X_{n+1}) - G(X_n)|}{|X_{n+1}-X_n|},
\end{equation}
which gives us a sequence of estimates $\{E_n\}$. Within the sequence, we treat the largest value as an upper bound for $b_2$, denoted as $\overline{b_2}$.

Using similar arguments in~\Cref{lem:xn+},  for sufficiently large $n$, we have $\Pi_n\subseteq \Omega_{sc}$ and $\Pi_n$ is nested between two ellipsoids centered at $x^*$ whose quadratic forms are defined by $\nabla^2 f(x^*)$. We also have
\begin{equation}\label{eq:ineq}
\lim_{n\rightarrow\infty}  |\Pi_n|^{-2}  \left( 2 (g_n - \fkf^*) / c_0^2 \right)^{d} =  \det(\nabla^2 f(x^*)) = \lambda_1\ldots \lambda_d \leq \lambda_1 \lambda_d^{d-1}.
\end{equation}
Since $\lambda_d \leq b_2 \leq \overline{b_2}$, %
we may use $\mathcal{E}_n$ as a lower bound estimation for $\lambda_1$ where
\begin{equation}\label{eq:b1_est}
\lambda_1 \geq  \mathcal{E}_n = (\overline{b_2})^{1-d} |\Pi_n|^{-2} \left( 2(g_n - \fkf^*) / c_0^2 \right)^{d} .
\end{equation}
Among the sequence $\{ \mathcal{E}_n \}$, we select the smallest value as a lower-bound estimate for $b_1$. We remark that in practice, if $\fkf^*$ is known, then we always have that $g_n > \fkf^*$ by construction. If $\fkf^*$ is unknown and needs to be estimated, one should use the minimum value in any finite sequence $\{g_n\}$. We also have $g_n > \fkf^*$ until the last element of the finite sequence, when the estimation~\eqref{eq:b1_est} should stop. We comment that if the function $f(x)$ satisfies our assumption, the ratio $(g_n - \fkf^*)^d/|\Pi_n|^2$ should be non-vanishing although both sides of the fraction go to zero as $n$ increases. For any given $g_n$, estimating the volume of $\Pi_n$ can be achieved by the Monte Carlo method; see~\Cref{sec:omega_n_est} for details.

We summarize the scheme of obtaining a lower bound $\underline{b_1}$ and an upper bound $\overline{b_2}$ into~\Cref{alg:b1b2}.

\begin{algorithm}[ht!]
\caption{Estimation of $\underline{b_1}$ and $\overline{b_2}$~\label{alg:b1b2}}
\begin{algorithmic}[1]
\STATE Given the initial iterate $X_0$, compute $X_1$ based on~\eqref{eq:diffdisc1}. Set $\underline{b_1} = \overline{b_2} = \frac{|G(X_0) - G(X_1)|}{|X_0 - X_1|}$.
\FOR{$n=2$ to $n_{\text{tot}}$}
\STATE Obtain $g_n$, (estimations of) $|\Pi_n|$ and $\fkf^*$. Given $X_{n-1}$, compute $X_n$  based on~\eqref{eq:diffdisc1}. 
\STATE Update $\overline{b_2} \leftarrow  \max\left(\overline{b_2}, \frac{|G(X_n) - G(X_{n-1})|}{|X_{n} - X_{n-1}|}\right)$ and $\underline{b_1} \leftarrow  \min\left(\underline{b_1},  \mathcal{E}_n \right)$ where $ \mathcal{E}_n $ follows~\eqref{eq:b1_est}.
\ENDFOR
\end{algorithmic}
\end{algorithm}

\section{Numerical Examples}\label{sec:numerics}

We have given analysis in the previous sections regarding the convergence of the proposed AdaVar algorithm~\eqref{eq:diffdisc1}. In~\Cref{sec:example1} below, we present a few simple numerical examples to quantitatively demonstrate the effectiveness of the proposed algorithm. The use of an adaptive state-dependent variance is very powerful, and we give two examples in~\Cref{sec:example2} where this technique is applied to settings that are not theoretically discussed in this paper. We will also discuss such settings in~\Cref{sec:conclusion} under future directions. In~\Cref{sec:estimate_fn}, we provide an example corresponding to the estimations discussed in~\Cref{sec:extension}.

\subsection{\texorpdfstring{The AdaVar Algorithm~\eqref{eq:diffdisc1}}{} Applied to Global Optimization Problems}\label{sec:example1}

\begin{figure}
\centering
\subfloat[$c = 0.05$]{
\includegraphics[width=0.95\textwidth]{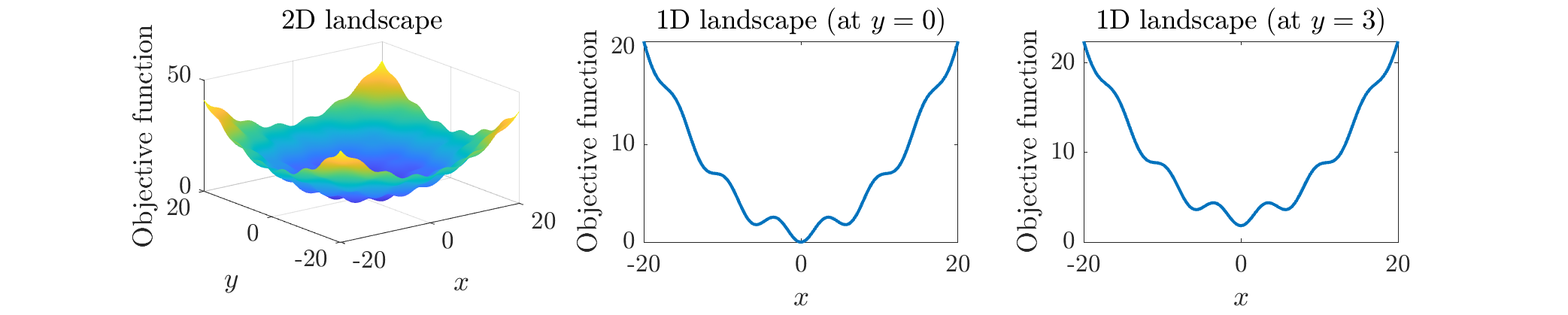} \label{fig:c005}}\\
\subfloat[$c = 0.01$]{
\includegraphics[width=0.95\textwidth]{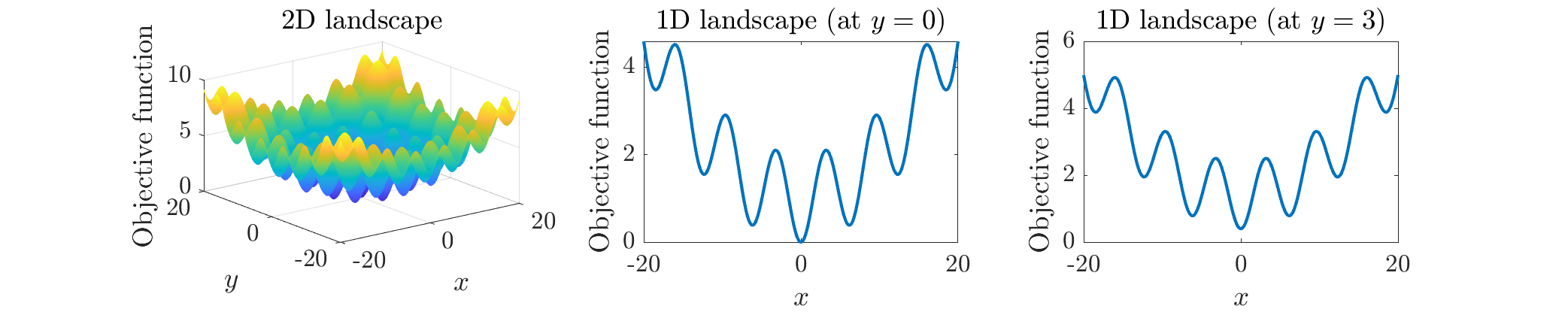}  \label{fig:c001}}
\caption{Optimization landscape of $J_1$ for $a=b=1$, $c = 0.05, 0.01$.}\label{fig:J1_c}
\end{figure}

\begin{figure}
    \centering
    \includegraphics[width=0.8\textwidth]{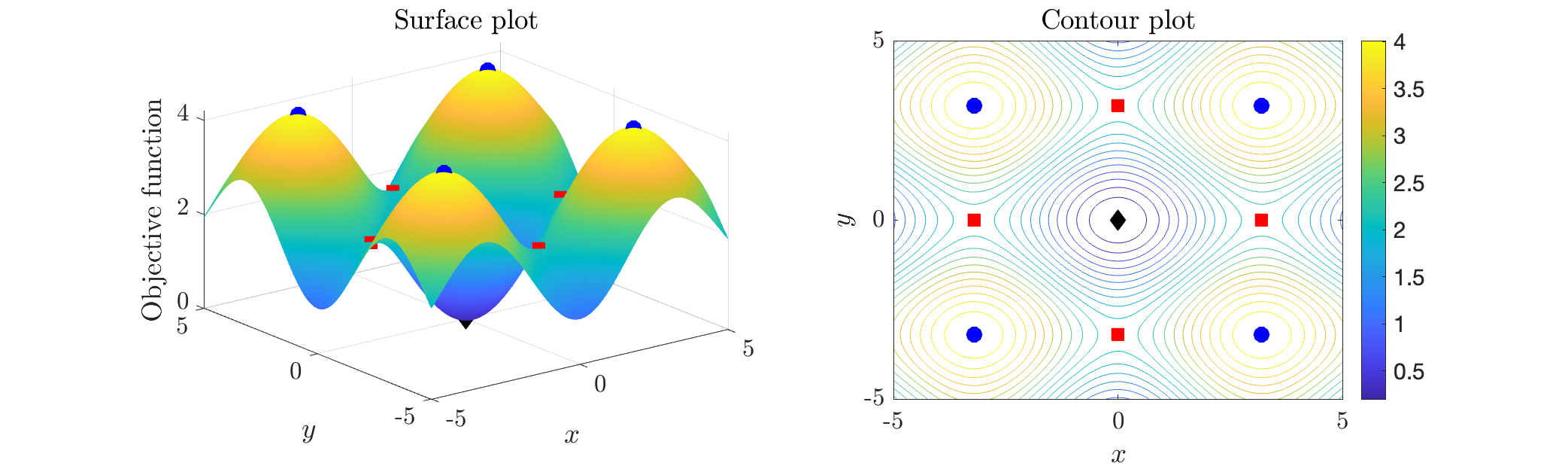}
    \caption{The zoom-in illustration of $J_1$ for $a=b=1, c = 0.01$ on the domain $[-5,5]^2$. The red squares represent the saddle points, the blue circles highlight the local minima and the black diamond marks the global minimizer.}
    \label{fig:saddle}
\end{figure}

In this subsection, we present a numerical optimization example to show the quantitative convergence of our proposed algorithm for finding the global minimizer of a highly nonconvex function. In the following set of tests, we consider a variation of the so-called Rastrigin function~\cite{rastrigin1974systems} as our objective function $J_1(\mathbf x)$ where $\mathbf{x} = (x_1,x_2,\ldots,x_d)^\top \in \R^d$.
\begin{equation} \label{eq:obj}
    J_1(\mathbf x) = a \left(d- \sum_{i=1}^d\cos(b x_1) \right)  + c\sum_{i=1}^d  x_i^2.
\end{equation}
When $c=0$, all local minima of $J_1$ are global minima. When $c>0$, $J_1$ has a unique global minimizer $x^* = (0,0)$ and infinitely many local minima and saddle points in $\R^d$. %
In terms of numerical implementation, we truncate our search domain $\Omega$ to be $[-20,20]^d$.

\begin{figure}
    \centering
    \includegraphics[width=0.95\textwidth]{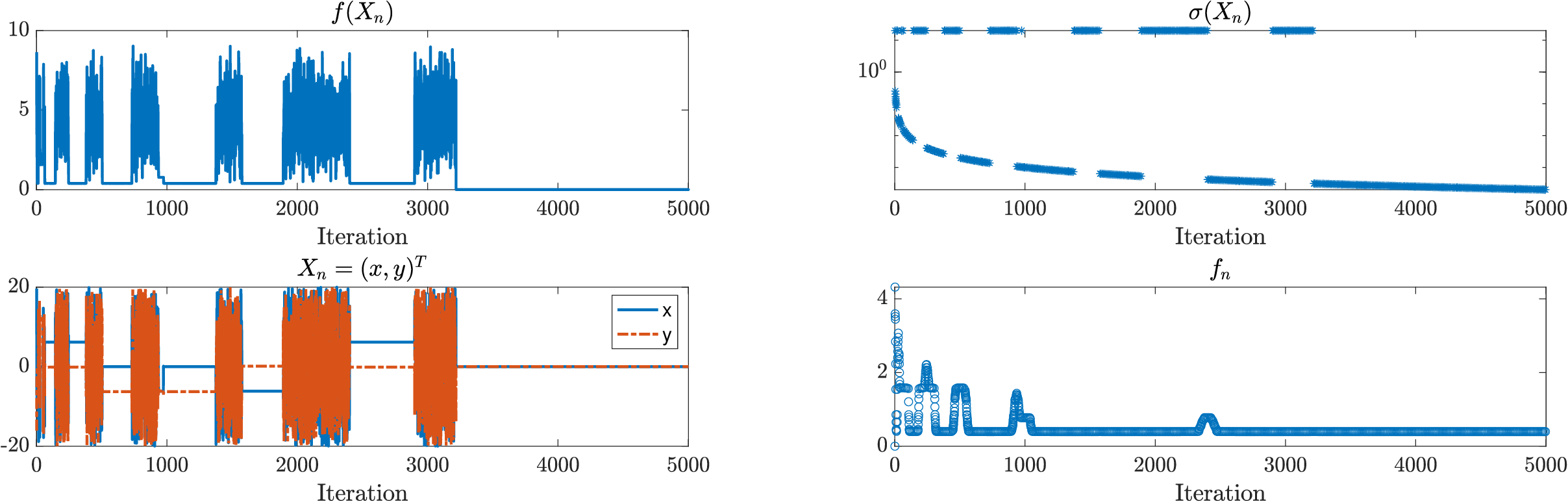}
    \caption{Convergence history of minimizing $J_1$ for $c= 0.01$ with the state-dependent noise following~\eqref{eq:test1_sigma}: the objective function value (top left), the standard deviation of the iterate (top right), the location of the iterates (bottom left), and the cutoff value that decides the sublevel set, i.e., the low-variance region (bottom right).}
    \label{fig:obj1D2-converge-single}
\end{figure}

\subsubsection{Two-Dimensional Objective Function} \label{sec:example_2_1}

First, we consider $d=2$. Fixing $a= b = 1$, we focus on two cases: $c = 0.05$ and $c = 0.01$. The search domain $\Omega = [-20,20]^2$; see~\Cref{fig:J1_c} for illustrations of the  optimization landscapes. In~\Cref{fig:saddle}, we highlight the global minimizer together with local minima and saddle points in the subset $[-5,5]^2$ for $c = 0.01$.

We use the proposed SGD algorithm to minimize $J_1$, which mainly follows~\eqref{eq:diffdisc1}. In terms of practical implementation, hyperparameters of the proposed algorithm could be treated as tunable parameters, like in many other optimization algorithms. Here, we set $\eta = 1$, ${\sigma}_0 = 1$ and $\widetilde{\sigma}_0 = 20$. 
We remark that the most important elements of the proposed algorithm are the state-dependent standard deviation $\sigma_n(f(x))$ of the noise term and the cutoff value $f_n$. We set
\begin{equation}\label{eq:test1_sigma}
    \sigma_n(f(x)) = \begin{cases} 
         {\sigma}_0 n^{-\alpha}, & f(x) <f_n,\\
         \widetilde{\sigma}_0,     & f(x)\geq f_n,
    \end{cases}
\end{equation}
where the cutoff value $f_n$ is the median ($0.5$ quantile) of the objective function values from the first $n$ iterations, $\{J_1(X_0), J_1(X_1), \dots, J_1(X_n)\}$. In this 2D example, we choose $\alpha = 1$. Without further information about $J(x)$, one could treat $\alpha$ as a tunable parameter. Due to the state-dependent bias imposed by~\eqref{eq:test1_sigma}, $\{X_n\}$ are more likely to stay at regions with small objective function values.

We present the convergence history of a single run of the algorithm in~\Cref{fig:obj1D2-converge-single} for the case $c= 0.01$. In \Cref{fig:obj1D2-converge}, we summarize the estimated probability of convergence with respect to the number of iterations and compare between the classical method~\eqref{eq:diffdisc} where $\sigma_n = 1/\sqrt{\log n}$ and our proposed method~\eqref{eq:diffdisc_new} where $\sigma_n$ is defined through~\eqref{eq:test1_sigma}. One can observe that the iterate $X_n$ converges to the global minimum $x^*$ in probability for both $c=0.01$ and $c=0.05$ using our proposed method, while with the classical method, there is no visible convergence. 
Quantitatively, using our proposed method, $\Prob(|X_n-x^*| < 0.01) \approx 1.000$ for $c=0.05$ and $\Prob(|X_n-x^*| < 0.01) \approx 0.997$ for $c=0.01$, while for the classical method, $\Prob(|X_n-x^*| < 0.01) \leq 1\%$ for both cases after $5\times 10^3$ iterations. The statistics are estimated by $10^3$ independent runs with the starting point $X_0$ uniformly sampled from the domain $\Omega$.

It takes more iterations for the case $c= 0.01$ than the case $c=0.05$ as there are more local minima and saddle points within the domain  $\Omega$ as $c$ decreases. Besides, the difference in the objective function value between the global minimum and the closest local minimum becomes smaller. Also, the difference in objective function value between the neighboring local minimum and maximum is much larger. The latter difference is often referred to as the energy barrier, commonly used in the context of phase transition and chemical reactions~\cite{ensing2005recipe}. Using our algorithm to induce phase transition in chemistry, thermodynamics, and other related fields could be one potential application of interest~\cite{weinan2006towards}.

\subsubsection{Ten-Dimensional Objective Function}\label{sec:example1_10D}
Next, we consider minimizing $J_1$ defined in~\eqref{eq:obj} for the case $d=10$. The number of local minima and saddle points in the search domain $\Omega = [-20,20]^d$ has increased exponentially in terms of $d$, demonstrating the curse of dimensionality in large-scale nonconvex optimization problems. We address that the basin of attraction is extremely small as it is approximately $10^{-10}$ of the entire search domain. In this 10D example, we consider $c = 0.03$ and $c= 0.05$ while fixing $a=b=1$ as before. Similar to the 2D case, we set $\eta = 1$, ${\sigma}_0 = 1$, and $\widetilde{\sigma}_0 = 20$. There is a slight difference in the setting of ${\sigma}_n$ due to a more complex optimization landscape for $d=10$: we choose $\alpha = 0.5$ in~\eqref{eq:test1_sigma} instead of $\alpha=1$ in the 2D case. The slower decay in ${\sigma}_n$ is to ensure that the iterates can visit the global basin of attraction before the variance becomes extremely small such that it becomes harder for the iterates to escape a local minimum. It also illustrates the flexibility by considering $\alpha$ as a tunable parameter in practical implementations. The cutoff value $f_n$ is still the median ($0.5$-quantile) of the set $\{ J_1(X_0), J_1(X_1), \dots, J_1(X_n)\}$.

The convergence histories of different setups using both the proposed method and the classical method are shown in~\Cref{fig:obj10D-converge}. One can observe the noticeably faster convergence of the proposed method as a result of the state-dependent variance. Compared to the 2D case, the 10D case requires more iterations to satisfy $\Prob(|X_n-x^*| < 0.01)$ due to the more complicated optimization landscapes of $J_1$ in higher dimensions. One can still observe the almost exponential convergence. Quantitatively, using our proposed method, $\Prob(|X_n-x^*| < 0.01) \approx 1.000$ for $c=0.05$ and $\Prob(|X_n-x^*| < 0.01) \approx 0.992$ for $c=0.03$, while for the classical method, $\Prob(|X_n-x^*| < 0.01) \approx 0$ in both cases within $10^5$ iterations. Again, the statistics are estimated by $10^3$ independent runs with uniformly sampled starting points. If we use uniform sampling to find an initial guess in the global basin of attraction, we may need approximately $\left(\frac{7}{40} \right)^{-10} \approx 3.7\times 10^7 $ different samples for the case $c=0.05$ in which the global basin of attraction is a subset of the hyper-cube $[-3.5, 3.5]^{10}$, while our domain is $[-20,20]^{10}$.

\begin{figure}
    \centering
    \includegraphics[width=0.8\textwidth]{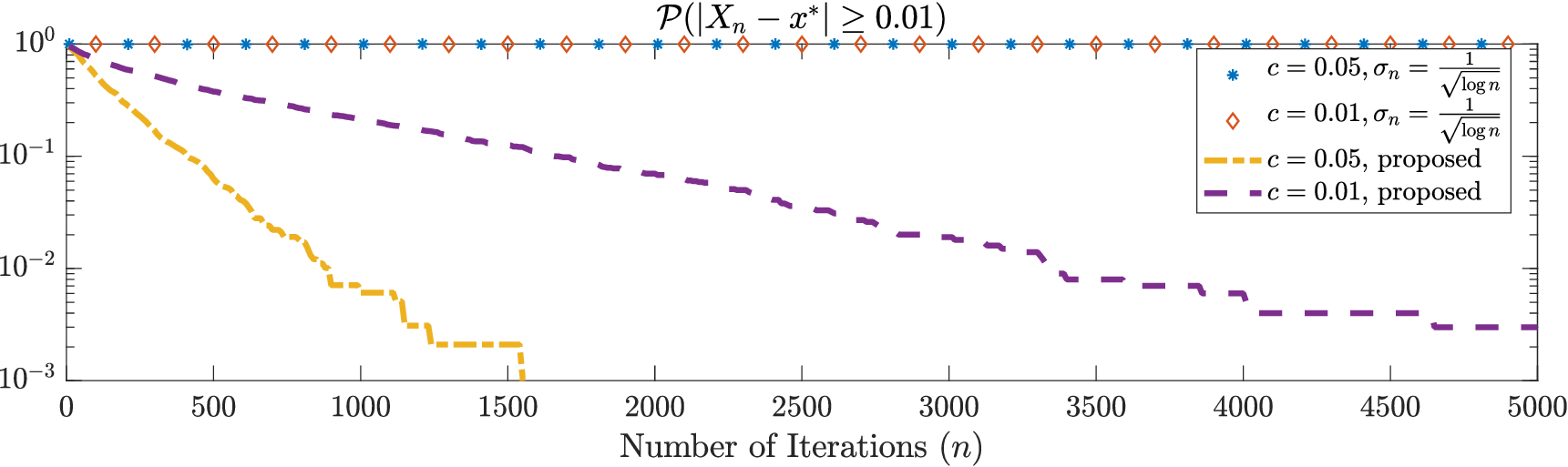}
    \caption{2D global optimization: estimation of $\Prob(|X_n-x^*|\geq 0.01)$ using $10^3$ independent runs where $X_n$ follows the classical method~\eqref{eq:diffdisc} and the proposed SGD algorithm~\eqref{eq:diffdisc_new} with variance defined in~\eqref{eq:test1_sigma}, respectively. The objective function is $J_1$ defined in~\eqref{eq:obj} with $d=2$ and $c=0.05$ or $c=0.01$.}
    \label{fig:obj1D2-converge}
\end{figure}

\begin{figure}
    \centering
    \includegraphics[width=0.8\textwidth]{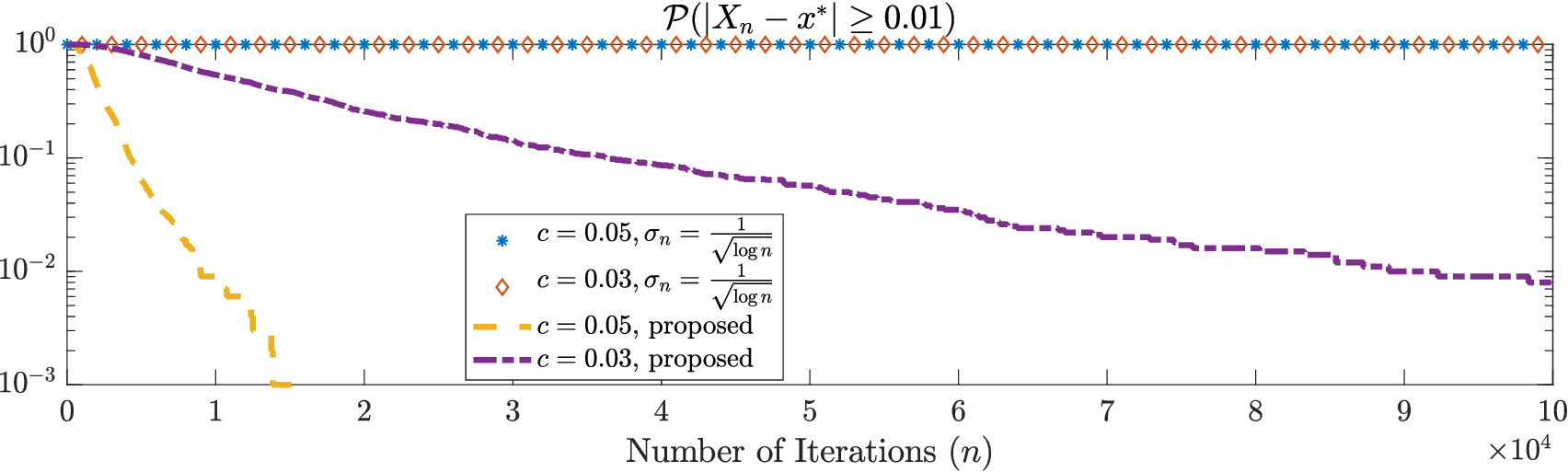}
    \caption{10D global optimization: estimation of $\Prob(|X_n-x^*|\geq 0.01)$ using $10^3$ independent runs where $X_n$ follows the classical method~\eqref{eq:diffdisc} and the proposed SGD algorithm~\eqref{eq:diffdisc_new} with variance defined in~\eqref{eq:test1_sigma}, respectively. The objective function is $J_1$ defined in~\eqref{eq:obj} with $d=10$ and $c=0.05$ or $c=0.03$.}
    \label{fig:obj10D-converge}
\end{figure}

\subsection{State-Dependent Variance Techniques Applied to Other Stochastic Descent Problems}\label{sec:example2}
This subsection presents two numerical examples based on more general settings than the one discussed earlier in this paper. One deals with stochastic objective functions as mentioned on the connection to formula~\eqref{eq:MLSGD101}—the other concerns gradient-free optimization. The promising numerical results demonstrate the power of adaptive state-dependent variance and indicate the potential for generalizing this work.

\subsubsection{Sampling of the Objective Function}\label{sec:ML_setup}
One very important application of stochastic gradient descent is machine learning. The random component is then not decoupled from the objective or loss function. Although this was not the setting proved in~\Cref{sec:proof}, we present a simple numerical example below showing the favorable outcomes in the case of sampling. 

As we have introduced in~\Cref{sec:intro}, the stochastic gradient descent method used in machine learning applications often has the form of~\eqref{eq:MLSGD101} where the variance is encoded in the estimated gradient $G(X_n;\xi)$ and $\xi$ is a random variable. For example, we are interested in finding the global minimum of the following optimization problem
\begin{equation}\label{eq:batch}
    \min_{x} \mathbb{E}_{\zeta}[J(x;\zeta)] \approx \frac{1}{M}\sum_{i=1}^M J(x;\zeta_i) = \widetilde{J}(x),
\end{equation}
where $M$ is the size of the training set (or sampled data) and $\{\zeta_i\}$ are i.i.d.~samples following the distribution of $\zeta$. For each $i$, the gradient $G(x;\zeta_i) = \nabla_x J(x;\zeta_i)$. In this case, we only have access to an approximated loss function $\widetilde J(x)$ and an estimated gradient
\begin{equation*}
    G(x;\xi_M)  = \frac{1}{M}\sum_{i=1}^M G(x;\zeta_i)  = \mathbb{E}_{\zeta} [ G(x;\zeta)] + \xi_M = G(x) + \xi_M,
\end{equation*}
where the random variable $\xi_M$ characterizes the random error in the gradient estimation. The relationship between $\zeta$ and $\xi_M$ depends on the concrete form of $J(x;\zeta)$. Thus,~\eqref{eq:MLSGD101} becomes
\begin{equation*}
    X_{n+1} = X_n - \eta_n \mathbb{E}_{\zeta} [ G(X_n;\zeta)]  + \eta_n \xi_M =  X_n - \eta_n G(X_n)  + \eta_n \xi_M,
\end{equation*}
where the noise term $\eta_n \xi_M$ is analogous to the $\sigma_n \psi_n$ term in~\eqref{eq:diffdisc_new}, which we wish to control in our proposed algorithm. As discussed in~\cite{hu2019diffusion}, changing the batch size $M$ is an effective way to control the power of the noise term given a fixed step size (see~\ref{itm:B1}).

\begin{figure}
    \centering
    \includegraphics[width=0.8\textwidth]{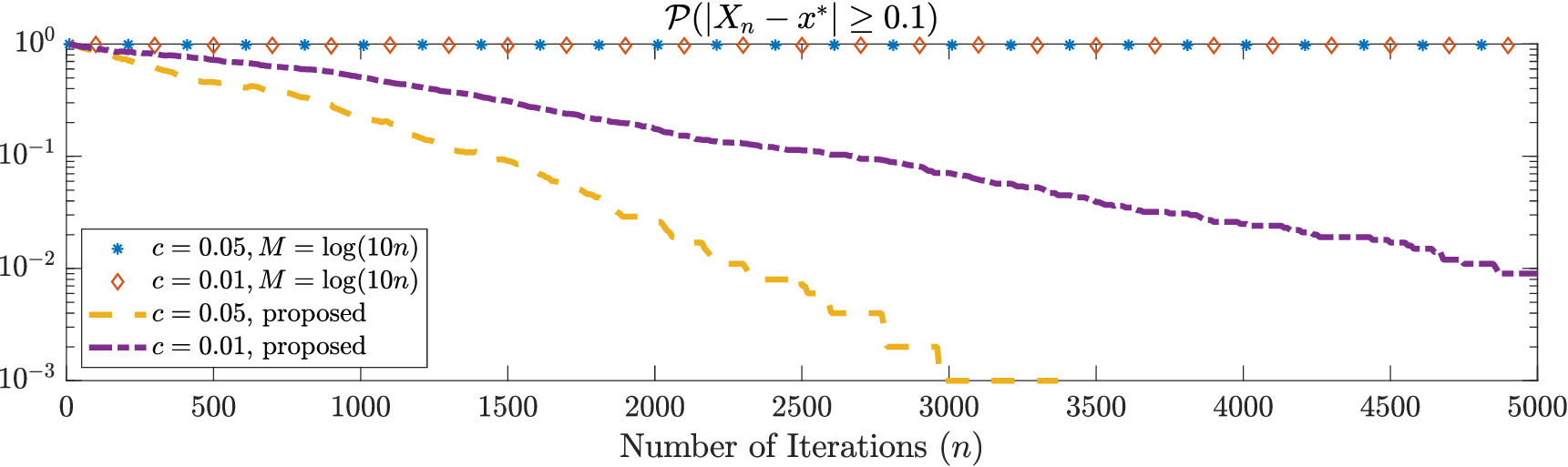}
    \caption{2D global optimization where the loss function is estimated by sampling (see~\Cref{sec:ML_setup}): the semi-log plots of $\Prob(|X_n-x^*|\geq 0.1)$ using the classical method~\eqref{eq:diffdisc} (batch size version) and the proposed algorithm~\eqref{eq:diffdisc2} with a state-dependent batch size to optimize $J_2$ defined in~\eqref{eq:J2} where $a=b=1$, $d=2$, $c=0.01$ and $c=0.05$.}
    \label{fig:obj2D-Batch-converge}
\end{figure}

Next, we examine our proposed algorithm~\eqref{eq:diffdisc1} with the control over the variance implicitly done through the adaptive batch size $M_n$. Consider 
\begin{equation}\label{eq:J2}
    J_2(\mathbf x;\zeta) = J_1(\mathbf x) + \zeta^1 \sum_{i=1}^d \sin(10 x_i) + \zeta^2 \sum_{i=1}^d \cos(10 x_i),\quad \zeta = [\zeta^1, \zeta^2]^\top,
\end{equation}
where $J_1$ is defined in~\eqref{eq:obj}, $d=2$ and $\zeta^1,\zeta^2\sim \mathcal{N}(0,0.5)$. Again, we fix $a=b=1$ in $J_1$ and then test cases where $c=0.01$ and $0.05$. Note that $J_1(\mathbf x) = \mathbb{E}_\zeta [J_2(\mathbf x;\zeta)]$. We modify our proposed algorithm in the following way in terms of the batch size $M_n$:
\begin{equation}\label{eq:diffdisc2}
    X_{n+1} = \begin{cases}  X_n -  \frac{\eta_n}{M_n}\sum_{i=1}^{M_n} G(X_n;\zeta_i)  , & \frac{1}{M_n}\sum_{i=1}^{M_n}J_2(X_n;\zeta_i) <f_n,\\
    \phi_{n}, &\frac{1}{M_n}\sum_{i=1}^{M_n}J_2(X_n;\zeta_i)  \geq f_n,\end{cases}
\end{equation}
where $M_n\in\N$ changes adaptively with $n$.  Other notations follow earlier definitions in~\eqref{eq:diffdisc1}. To mimic the variance decay in the classical method~\eqref{eq:diffdisc}, the batch size should follow $M_n = \mathcal{O}(\log n)$~\cite{hu2019diffusion}. For our proposed method~\eqref{eq:diffdisc1}, we may choose $M_n = \mathcal{O}(n^{2\alpha})$ where $\alpha$ is the constant in~\ref{itm:B3}, but could be regarded as a tunable parameter in practical implementations. For the numerical test shown in~\Cref{fig:obj2D-Batch-converge}, we choose $M_n = \log (10 n)$ for the classical method and $M_n = n$ for our proposed algorithm, with other parameters following the same setup as in Section~\ref{sec:example_2_1}.

In~\Cref{fig:obj2D-Batch-converge}, the proposed algorithm~\eqref{eq:diffdisc2} yields an exponential convergence. As expected, the case $c=0.05$ still converges faster than the case $c=0.01$. On the other hand, despite the theoretical guarantee of global convergence, the classical $M_n=\mathcal{O} (\log n)$ decreases the variance very slowly and is thus unable to give any visible convergence within $5\times 10^3$ iterations. Comparing with~\Cref{fig:obj1D2-converge}, the batch size-based setup yields slower convergence because not only the noise is no longer isotropic Gaussian as a result of its coupling with the gradient formulation, but also the loss function value now becomes an estimation $\widetilde{J}(X_n) = \frac{1}{M_n}\sum_{i=1}^{M_n}J_2(X_n;\zeta_i)$. An important aspect of our proposed algorithm is to use the loss function value as an indicator to decide if $X_n\in\Omega_n$.  In this example, the iterates are thus affected by the random error between $J(X_n)$ and $\widetilde{J}(X_n)$, particularly so when close to the ground truth $x^*$. A slower decay rate of $f_n$ is preferable for such situations.

\subsubsection{Gradient-Free Optimization}\label{sec:grad_free}

In another study~\cite{engquist2023adaptive}, we delved deeply into the scenario where the gradient term in~\eqref{eq:diffdisc2} is omitted and analyzed its convergence characteristics when the state-dependent diffusion coefficient is appropriately controlled. In this section, we present a straightforward example to numerically showcase how the state-dependent diffusion by itself can effectively concentrate over the global minimum with the goal of achieving global optimization.

We consider a modified algorithm compared to~\eqref{eq:diffdisc_new} where the gradient component is removed,
\begin{equation}\label{eq:SDE_nograd}
    X_{n+1} = X_n + \sigma_n(f(X_n)) \psi_n,
\end{equation}
and enforce the periodic boundary condition for iterates that exit the domain $\Omega = [0,4]$. Let $\sigma_n(f(x))$ be a  piecewise-constant function where
\begin{equation}\label{eq:diff_piece}
    \sigma_n(f(x)) = \begin{cases}
 s, & x\in \widetilde \Omega,\\
1/s, &\text{otherwise.}
    \end{cases}
\end{equation}
for some $s>0$. Here, $\widetilde \Omega$ mimics the low-variance sublevel set $\Omega_n$ in earlier parts of the paper.

\begin{figure}
    \centering
    \includegraphics[width=\textwidth]{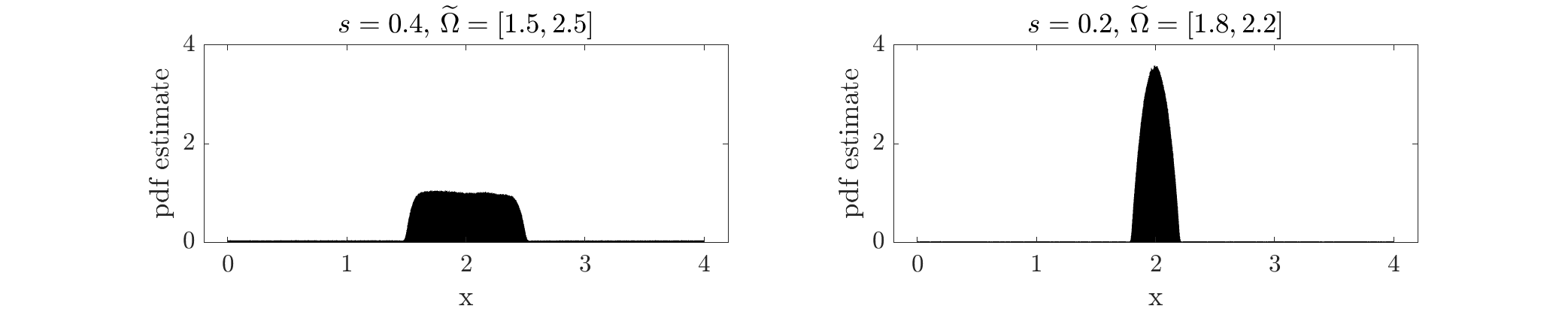}
    \caption{Approximated density function of the invariant measure for~\eqref{eq:SDE_nograd} where the noise term follows~\eqref{eq:diff_piece}. Left: $s=0.4$, $\widetilde \Omega = [1.5,2.5]$; Right: $s=0.2$, $\widetilde \Omega = [1.8,2.2]$.}
    \label{fig:all-s}
\end{figure}

Starting with $X_0=1$, we compute $N=2\times 10^7$ iterations.
We approximate the occupation measure 
\[
    \mu_{X_0,T}(B)
    =\frac{1}{N} \sum_{n=1}^N \mathbbm{1}_{X_n\in B},
\]
where $B\subset \R^d$ is any Borel measurable set and $\mathbbm{1}$ is the indicator function. Under mild conditions, $\mu_{X_0,T}$ converges to the physical invariant measure $\mu^*$~\cite{young2002srb}. We estimate $\mu^*$ using $N$ iterates; see~\Cref{fig:all-s} for the estimated probability density functions based on two setups for the variance.

When $s=0.4$ and $\widetilde \Omega = [1.5, 2.5]$, the stationary distribution concentrates on the small-variance region $\widetilde \Omega$ while the density becomes more concentrated for the case of $s=0.2$ and $\widetilde \Omega = [1.8, 2.2]$, as seen in~\Cref{fig:all-s}. %
A comparison between the two plots shows the evident impact of the state-dependent stochasticity based on which the iterates are more likely to stay within the region of small variance. 

In our proposed SGD algorithm~\eqref{eq:diffdisc_new} with the variance~\eqref{eq:two-stage-sigma}, in particular, algorithm~\eqref{eq:diffdisc1}, we utilize such property to align the small-variance region with the sublevel set of the objective function, which we aim to find the global minimizer. On the one hand, the iterate is more likely to stay within the sublevel set due to the small variance. On the other hand, as the iteration continues, the small-variance region shrinks as we decrease the cutoff value that decides the sublevel set. Finally, the iterate converges to the global minimizer in probability due to both aspects under properly chosen decay rates.

\subsection{Estimating \texorpdfstring{$\Omega^*(\fkf)$}{}}
\label{sec:estimate_fn}

\begin{figure}
    \centering
\subfloat[Estimation of the relation between $f_n$ and $|\Omega_n|$ using i.i.d. samples]{\includegraphics[width=0.8\textwidth]{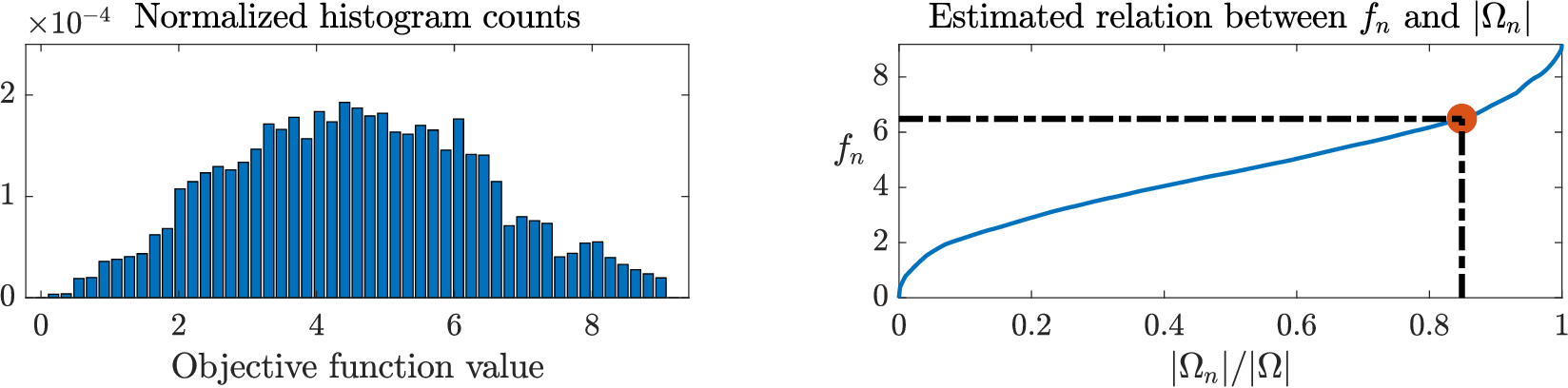}\label{fig:fniid}}\\
  \subfloat[Estimation of the relation between $f_n$ and $
 |\Omega_n|$ using high-variance iterates]{\includegraphics[width=0.8\textwidth]{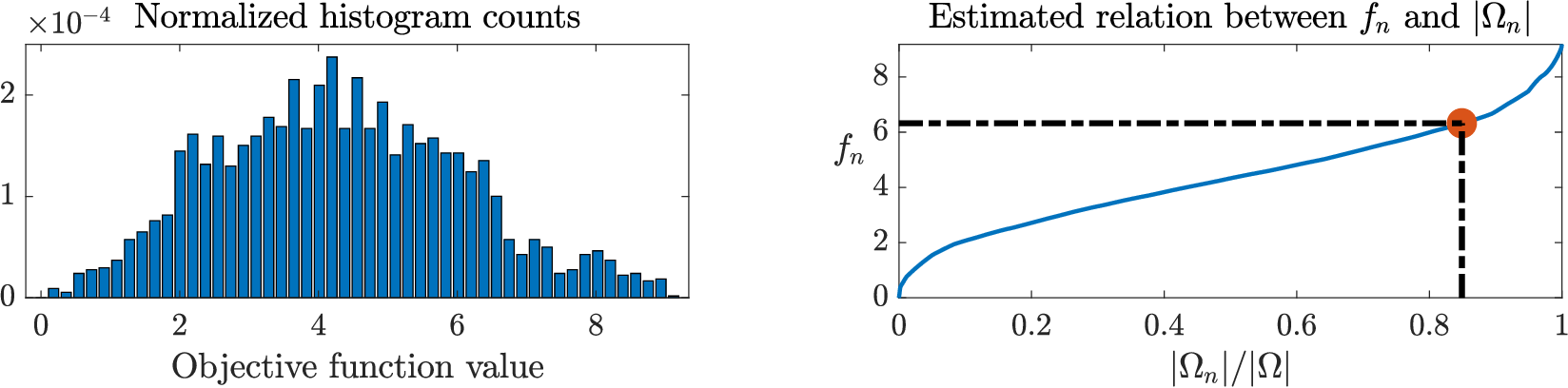}\label{fig:fnXn}}\\
  \subfloat[The comparison between the two estimations]{\includegraphics[width=0.8\textwidth]{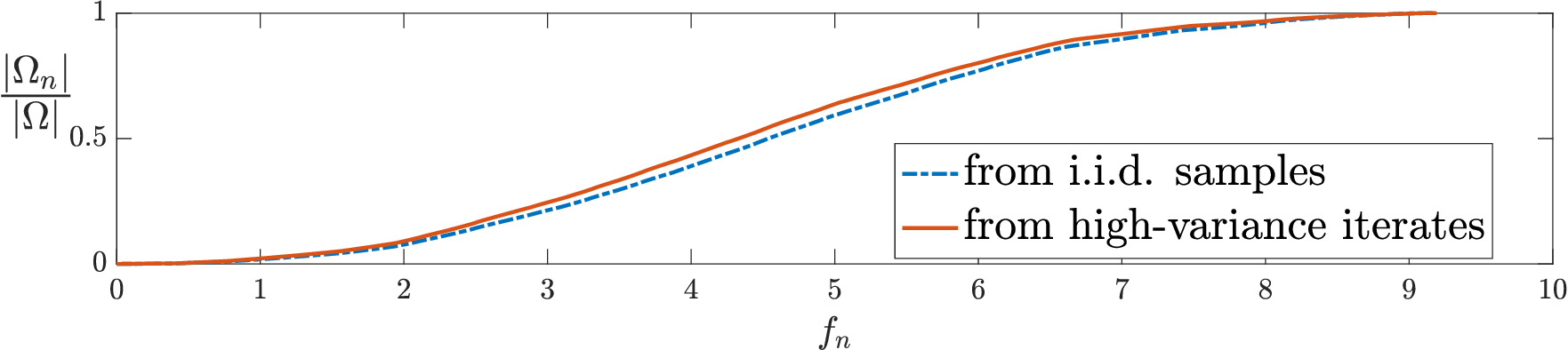}\label{fig:cdf-compare}}
    \caption{(a)~Estimation of the relation between $f_n$ and $
 |\Omega_n|$ using i.i.d. samples; (b)~Estimation of the relation between $f_n$ and $
 |\Omega_n|$ using high-variance iterates; (c)~A pointwise comparison between the two estimations of $|\Omega_n|$ as a monotonically increasing function of $f_n$.}
    \label{fig:fn_Omega_n}
\end{figure}

We illustrate the effectiveness of using high-variance samples to estimate $\Omega^*(\fkf)$ through numerical examples. This corresponds to the earlier discussion in Section~\ref{sec:omega_n_est}. Recall that $\Omega^*: \R \mapsto \R^+$ and $\Omega^*(\fkf) = |\{x\in \Omega|f(x)\leq \fkf \}|$. It is a function that maps the cutoff objective function value $f$ to the volume of the $\fkf$-sublevel set of the objective function $f(x)$. Again, we use $J_1$ in~\eqref{eq:obj} as our objective function by setting $a=b=1, c=0.01, d=2$. In~\Cref{fig:fniid}, we draw $N = 1\times 10^9$ i.i.d.\ samples $\{Y_i\}$ to produce a baseline result. The left-hand side plot is a histogram of $\{J_1(Y_i)\}$, which we could interpret as the empirical probability density function (pdf) for the random variable $J_1(x)$ where $x\sim\mathcal{U}([-20,20]^2)$. 

The right-hand side plot in~\Cref{fig:fniid} is the corresponding estimated inverse cumulative distribution function (ICDF). In our case, the ICDF illustrates how the cutoff value $f_n$ depends on $|\Omega_n|/|\Omega|$. Given the approximated ICDF, one can use numerical interpolation to obtain a sequence of cutoff values $\{f_n\}$ such that the volume of the sublevel set $|\Omega_n|$ decreases following a desired rate.

Next, we use about $5.3\times 10^5$ high-variance iterates $\{X_{n_k+1}\}$ ($\sigma_{n_k}(f(X_{n_k})) = \widetilde{\sigma}_{0} = 20$), generated as part of the SGD algorithm, to estimate the relation between $f_n$ and $|\Omega_n|$. Note that we set $\widetilde{\sigma}_{0} = 20$ in~\Cref{eq:test1_sigma} for numerical illustration. We also remark that we do not need to store the iterates, but only their objective function values (real scalars). The similar two plots are shown in~\Cref{fig:fnXn}. For $|\Omega_n| = 0.85|\Omega|$, the estimated $f_n$ is $6.4855$ using the i.i.d.\ samples, and $6.3233$ using the high-variance iterates, as highlighted in both figures. A pointwise comparison of the estimated cumulative distribution function ($|\Omega_n|/|\Omega|$ as a function of $f_n$) based on the two approaches is presented in~\Cref{fig:cdf-compare}.

\section{Conclusion and Future Directions}\label{sec:conclusion}

In this paper, we have proposed AdaVar, a stochastic gradient algorithm~\eqref{eq:diffdisc_new} where the variance of the noise term is adaptive and state-dependent, with the goal of finding the global minimum of a nonconvex optimization problem. The algorithm is proven to have an algebraic rate of convergence in~\Cref{sec:proof} under quite general assumptions, and this is a substantial improvement over the current standard logarithmic rate. The main component of the proposed algorithm is to allow adaptivity of the variance so that it changes not only in time but also in the location of the iterates so that a stronger noise is imposed when the iterate yields a larger objective function value; see~\eqref{eq:two-stage-sigma} for an example. If we know enough information about the objective function, the proposed algorithm costs the same as the classical method of stochastic gradient descent.

The techniques for accurately tuning parameters in the algorithms add very little to the overall computation cost. Several numerical examples presented in~\Cref{sec:numerics} further demonstrate the efficiency of the algorithm. 

There are quite a few future directions along the line of algorithms with state-dependent variance, which we here call AdaVar. Let us remark on the two briefly introduced and tested examples in~\Cref{sec:example2}. First is the application to stochastic objective functions in~\Cref{sec:ML_setup}. As mentioned in~\Cref{sec:intro}, the focus of this paper is on minimizing a given deterministic function and then adding a controlled stochastic component to achieve global convergence. A natural extension is the case where the stochastic component is directly linked to the objective function. This can either be because the objective function is intrinsically stochastic or, as is common in machine learning, a random sampling step is added to reduce the computational cost. There are two new challenges in this regime. One is the reduced control of the variance. In the machine learning case, natural tools include adjusting the batch size and controlling the adaptive step size, or the so-called learning rate. The other challenge is that only sampled values of the objective function are known. The sampling values can still be used, or different forms of aggregations or averaging could potentially be more representative.

Second, as illustrated in~\Cref{sec:grad_free}, the steady-state distribution of a pure diffusion without the gradient term is also subject to the adaptivity of the noise variance as a function of the location. Therefore, it is possible to design a zeroth-order stochastic descent method for global optimization while embedding the information of the objective function into the noise term. In our latest research~\cite{engquist2023adaptive}, we focused on this area and demonstrated convergence in both spatial and probabilistic terms at an algebraic rate with simple derivative-free stochastic algorithms. This approach can be further probed in the context of machine learning, as highlighted in the previously mentioned future prospects.

Another direction is to replace the use of an estimated $\Omega_n$ in determining the cutoff value $f_n$ by using the escape rates from $\Omega_n$ to $\Omega_n^c$ in control-type algorithms. Finally, it is important to apply the algorithms in several realistic settings to adjust implementation details and to establish best practices.

\section*{Acknowledgment}
This work is partially supported by the National Science Foundation through grants DMS-1620396, DMS-1620473, DMS-1913129, DMS-1913309, DMS-2110895, DMS-1937254, DMS-2309802 and DMS-2409855, and by the Office of Naval Research through grant
N00014-24-1-2088.  Y.~Yang acknowledges support from Dr.~Max R\"ossler, the Walter Haefner Foundation, and the ETH Z\"urich Foundation.  This work was done in part while Y.~Yang was visiting the Simons Institute for the Theory of Computing in Fall 2021. The authors thank the referees for carefully reading our manuscript, providing comments and suggestions.

\bibliographystyle{plain} 
\bibliography{references}

\end{document}